\documentclass[a4paper, 11pt]{article}

\usepackage{amsmath,amssymb,amsthm}
\usepackage[latin1]{inputenc}
\usepackage{version,tabularx,multicol}
\usepackage{graphicx,float,psfrag}
\usepackage{stmaryrd}
 \usepackage{enumerate}
 \usepackage{color}
\usepackage[pdftex]{hyperref}
\usepackage{authblk}

\usepackage[numbers,sort&compress]{natbib}

\theoremstyle{plain}
\newtheorem{theorem}{Theorem}[section]

 \newtheorem{corollary}[theorem]{Corollary}

 \newtheorem{proposition}[theorem]{Proposition}
\newtheorem{lem}[theorem]{Lemma}
 \newtheorem{lemma}[theorem]{Lemma}

\newtheorem{remark}{Remark}[section]

\newcommand{\vt}{{\Vert}}

\newcommand{\ii}{{\mathrm i}}

\newcommand{\bn}[1]{\emph{\textbf{#1}}}
\newcommand{\bm}[1]{\mbox{\boldmath $#1$}}

  \makeatletter
  \@addtoreset{equation}{section}
  \makeatother

 \def\beqlb{\begin{eqnarray}}\def\eeqlb{\end{eqnarray}}
 \def\beqnn{\begin{eqnarray*}}\def\eeqnn{\end{eqnarray*}}

 \def\mbb{\mathbb}

\newcommand{\bcen}{\begin{center}}
\newcommand{\ecen}{\end{center}}
\newcommand{\bgeqn}{\begin{equation}}
\newcommand{\edeqn}{\end{equation}}

\def\M{{\mathcal  M}}

\begin{document}

\title{Lower deviation probabilities for supercritical multi-type Galton--Watson
	processes}

\author{  Jiangrui Tan\thanks{%
 School of Mathematics and Statistics, Central South University, Changsha 410075, P.R. China. Email:tanjiangrui@csu.edu.cn
} }
\maketitle

\noindent{\bf Abstract}\quad
This paper provides a detailed analysis of the lower deviation probability properties for a $d$-type ($d>1$) Galton--Watson (GW) process $\{\bm{Z}_n=(Z_n^{(i)})_{1\le i\le d};n\ge0\}$ in both Schr\"{o}der and B\"{o}ttcher cases. We establish explicit decay rates for the following probabilities: $$\mbb{P}(\bm{Z}_n=\bm{k}_n),~ \mbb{P}(|\bm{Z}_n|\le k_n), ~\mbb{P}(Z^{(i)}_n=k_n)~~\text{and}~~\mbb{P}(Z^{(i)}_n\le k_n), 1\le i \le  d,$$ respectively, where $\bm{k}_n\in\mbb{Z}_+^d$, $|\bm{k}_n|=\mathrm{o}(c_n)$, $k_n=\mathrm{o}(c_n)$ and $c_n$ characterizes the growth rate of $\bm{Z}_n$. These results extend the single-type lower deviation theorems of Fleischmann and Wachtel (Ann. Inst. Henri Poincar\'e Probab. Statist.\textbf{43} (2007) 233-255), thereby paving the way for analysis of precise decay rates of large deviations in multi-type GW processes.  

 \vspace{0.3cm}

\noindent{\bf Keywords}\quad Multi-type Galton--Watson
processes; Lower deviations; Local limit theorem; Multi-dimensional Cramér transform

\noindent{\bf MSC}\quad Primary 60J80; Secondary 60F10.\\[0.4cm]

\bigskip

\section{Introduction}
\vspace{3mm}
We begin by reviewing key results concerning the lower deviations of a single-type supercritical Galton--Watson (GW) process. Let $Z=\{Z_n;n\geq0\}$ with $Z_0\equiv1$ denote a supercritical non-degenerate GW process, that satisfies $m:=\mbb{E}Z_1>1$ and $\mbb{P}(Z_1=j)<1$ for any $j\in \mbb{Z}_+$. The offspring generating function is defined as $f(s)=\mbb{E}s^{Z_1}=\sum_{k\ge0} p_ks^k$.\par
As established in classical literature (see, e.g., [1]), there exists a sequence of normalization constants  $\{c_n;n\ge1\}$ such that
\begin{equation}\label{cn0}
	c_n^{-1}Z_n\overset{\text{a.s.}}{\longrightarrow} W, ~\text{as}~n\to\infty. 
\end{equation} 
where $W$ is a non-degenerate limit random variable. Specifically, the normalization sequence can be expressed as $c_n=m^nL(m^n)$, where $L(\cdot)$ is a slowly varying function at $\infty$. Notably, $\lim_{x\to\infty}L(x)$ exists and is positive if and only if the condition $\mbb{E}Z_1\log Z_1<\infty$ holds. 

Furthermore, the limit variable $W$ has a strictly positive continuous density $w(\cdot)$ on $(0,\infty)$. That is, for any $0<a<b<\infty$, $\mbb{P}(a\le W\le b)=\int_{a}^bw(t)\mathrm{d}t$. This leads to the global limit theorem:
\begin{equation}\label{global}
	\lim_{n\to\infty}\mbb{P}(c_n^{-1}Z_n\ge a)=\int_a^{\infty}w(t)\mathrm{d}t,~a>0.
\end{equation}
To state lower deviation results in the single-type case, we need the following definitions: 
\begin{itemize}
	\item The offspring generating function $f$ is of type $(d,u)$ if $d\geq1$ is the greatest common divisor of the set $\{j-i:j\neq i, p_ip_j>0\}$ and $u=\min\{i:p_i>0\}$.
	\item $f$ is said to be Schr\"{o}der type if $p_0+p_1>0$, and B\"{o}ttcher type if $p_0+p_1=0$.
\end{itemize}
 Therefore, Fleischmann and Wachtel's results \cite[Theorem 4 \& 6]{lower} on single-type lower deviations can be stated as follows:
\begin{theorem}[Schr\"{o}der case]
Let the offspring generating function $f$ is of type $(d,u)$ and define $a_k=\min\{l\geq1:c_l\ge k\}$. Then for any $k_n\le c_n$ satisfying $k_n\to\infty$,
\[
\sup_{k \in [k_n, c_n]} \left| \frac{m^{n-a_k} c_{a_k}}{dw(k/(m^{n-a_k} c_{a_k}))} \mathbb{P}(Z_n = k) - 1 \right| \underset{n \uparrow \infty}{\longrightarrow} 0
\]
and
\[
\sup_{k \in [k_n, c_n]} \left| \frac{\mathbb{P}(0 < Z_n \leq k)}{\mathbb{P}(0 < W < k/(m^{n-a_k} c_{a_k}))} - 1 \right| \underset{n \uparrow \infty}{\longrightarrow} 0.
\]
When $\mbb{E}Z_1\log Z_1<\infty$ holds, then $c_{a_k}=m^{a_k}$ and we have the $a_k$-free formula:
\begin{equation}\label{ems2}
	\mbb{P}(Z_n=k)=dm^{-n}w(k/m^n)(1+\mathrm{o}(1)).
\end{equation}
\end{theorem}

\begin{theorem}[B\"{o}ttcher case]
Let the offspring generating function $f$ is of type $(d,u)$ and define $b_n=\min\{l\geq1:c_lu^{n-l}\ge2k_n\}$. Then there exists positive constants $B_1$ and $B_2$ such that for all $k_n\equiv u^n (\text{mod}~d)$ with $k_n\ge u^n$ but $k_n=\mathrm{o}(c_n)$,
\begin{align*}
	-B_1 &\leq \liminf_{n \uparrow \infty} \mu^{b_n - n} \log \left[ c_n \mathbb{P}(Z_n = k_n) \right]\\
	&\leq \limsup_{n \uparrow \infty} \mu^{b_n - n} \log \left[ c_n \mathbb{P}(Z_n = k_n) \right] \leq -B_2.
\end{align*}
The above inequality remains true if $c_n \mathbb{P}(Z_n = k_n)$ is replaced by $ \mathbb{P}(Z_n \le k_n)$ and the assumption $k_n\equiv u^n (\text{mod}~d)$ is not required.
\end{theorem}
The above theorems give detailed decay rates of lower deviations of $Z_n$, i.e.,  $\mbb{P}(Z_n=k_n)$ with $k_n=\mathrm{o}(c_n)$. These results play a crucial role in the study of large deviation of the ratio of $Z_{n+1}/Z_n$ (see \cite{large,large2}), which is the well-known Lotka--Nagaev estimator of the offspring mean. 

In this paper, we obtain lower deviations and normal deviations for multi-type Galton--Watson process in both Schr\"{o}der and B\"{o}ttcher cases. These results extend naturally to multi-type processes, which parallel to Fleischmann and Wachtel's theorems in the single-type case. Existing literature has also extensively investigated the lower deviation probabilities for variant models of branching processes, such as GW processes in random environments. For specific details, see \cite{lower1,lower2,lower3}.

The remainder of this paper is structured as follows: Section 2 introduces known results on multi-type GW processes and presents our main theorems. Schr\"{o}der type and B\"{o}ttcher type asymptotic regimes are analyzed in Sections 3 and 4 respectively, with complete proofs provided for each case.

\section{ Preliminaries and main results}
\vspace{3mm}
\subsection{Notations}
The notation in this article is complex. We begin by establishing some general notations to be employed in this work.  Throughout this paper, we consider $d$-type ($d>1$) GW processes, and bold symbols often refer to $d$-dimensional row vectors or $d\times d$ matrices. Specifically:
\begin{itemize}
	\item $\bm{1}$, $\bm{0}$ and $\bm{e}_i$ represent vectors with all components being $1$, $0$ and $0$ except that $k$-th component being $1$, respectively.
	\item The $k$-th component of a vector $\bm{x}$ is written as $x^{(k)}$.
	\item Define the $\ell_{1}$ and $\ell_2$ norm of $\bm{x}$ as $|\bm{x}|=|\bm{1}\cdot\bm{x}|$ and $\Vert \bm{x}\Vert=(\sum_k (x^{(k)})^2)^{1/2}$, respectively. 
	\item Let $\bm{x}^T$ denote the transpose of $\bm{x}$, and define the component-wise operations: \[
	\bm{x}^{\bn{y}}=\prod_k (x^{(k)})^{y^{(k)}},~~\mathrm{e}^{\bm{x}}=(\mathrm{e}^{x^{(k)}})_{k\geq1},~~\log \bm{x}=(\log x^{(k)})_{k\ge1}.\]
	\item The partial orders $\bm{x}\le\bm{y}$ denote $x^{(k)}\leq y^{(k)}$ for all $k$, $\bm{x}<\bm{y}$ means the same but requiring strict inequality in at least one component $k$ and $\bm{x}\ll\bm{y}$ means $x^{(k)} < y^{(k)}$ for all $k$.
	\item Conditional expectations and probabilities are denoted by $\mbb{E}_i(\cdot)=\mbb{E}(\cdot|\bm{Z}_0=\bm{e}_i)$ and $\mbb{P}_i(\cdot)=\mbb{P}(\cdot|\bm{Z}_0=\bm{e}_i)$.

\end{itemize}
Notational conventions:
\begin{itemize}
	\item Without ambiguity, we will reuse some notations that appear in the previous section to emphasize their conceptual connections. 
	\item Symbols $C$, $C'$, $\delta$ and $\delta'$ (with possible subscripts) represent strictly positive constants that may vary line-by-line, unless otherwise specified.
	\item  The upright ``$\mathrm{i}$" specifically denotes the imaginary unit, while the italicized ``$i$" usually refers to the initial particle type.
	\item Big-O notation ($\mathrm{O}$) and little-o notation ($\mathrm{o}$) have their usual meanings in describing the asymptotic behavior of functions.
		\item Other symbols will be explicitly declared when first introduced.
\end{itemize}
\subsection{Classification and assumptions}
Consider a $d$-type Galton--Watson process $Z=\{\bm{Z}_n=(Z_n^{(i)})_{1\leq i\leq d};$ $n\geq0\}$ (see \cite[Chapter V]{Athreya1972Branching} for detailed definition) with offspring generating function $f^{(i)}(\bm{s})=\mbb{E}_i\bm{s}^{\bn{Z}_1}$ and
\begin{equation}\label{fs}
\bm{f}(\bm{s})=(f^{(1)}(\bm{s}),f^{(2)}(\bm{s}),\cdots,f^{(d)}(\bm{s})),\quad \bm{s}=(s_1,s_2,\cdots,s_d)\in[0,1]^d.\end{equation}
  Define $$
D_{ij}(\bm{s})=\cfrac{\partial\,f^{(i)}}{\partial\,s_j}(\bm{s}),\quad a_{ij}=D_{ij}(\bm{0}),\quad m_{ij}=D_{ij}(\bm{1}),
$$
and $\bm{A}=(a_{ij})_{1\le i,j \le d}$, $\bm{M}=(m_{ij})_{1\le i,j\le d}$. Clearly, $\bm{M}$ is the mean matrix of $\{\bm{Z}_n;n\geq0\}$. Denote the maximal eigenvalue of $\bm{M}$ by $\rho$.  Denote the $n$-th iterate of $\bm{f}$ by $\bm{f}_n$. \par We only consider the supercritical case, i.e., $\rho>1$. The normalized right and left eigenvectors of $\bm{M}$ are denoted by $\bm{u}$ and $\bm{v}$ respectively, satisfying $\bm{v},\bm{u}\gg\bm{0}$, $\bm{v}\cdot\bm{u}=1$ and $\bm{1}\cdot\bm{v}=1$. To avoid trivialities, we assume $\bm{f}(\bm{0})=\bm{0}$, and the process $\{\bm{Z}_n;n\geq0\}$ is nonsingular and positive regular, that is, $\bm{M}$ is irreducible and $\bm{f}(\bm{s})\neq \bm{M}\bm{s}^T$ component-wise. 

   Similarly to the single-type case, there exists a sequence of constants $\{c_n\}$ such that
\begin{equation}\label{cn}
	c_n^{-1}\bm{Z}_n\overset{\text{a.s.}}{\longrightarrow}~\text{some non-degenerate}~~W\bm{v} 
\end{equation} 
$\text{component-wise} ~\text{as}~n\to\infty$. The normalized sequence $\{c_n\}$ can be chosen to satisfy
\begin{equation}\label{cn1}
	c_0=1, c_n\le c_{n+1} \le \rho c_{n} ~\text{and} ~c_n=\rho^nL(\rho^n),~n\ge1,
	\end{equation}
   where $L(\cdot)$ varying slowly at $\infty$ and $\lim_{x\to\infty}L(x)$ is positive if and only if \begin{equation}\label{llogl}
   	\mbb{E}_i|\bm{Z}_1|\log |\bm{Z}_1|<\infty~~\text{for all}~~i.
   \end{equation} Moreover,  there exist positive continuous density functions $w_i,i>0$, such that $\mbb{P}_i(a\le W\le b)=\int_{a}^bw_i(t)\mathrm{d}t$ for any positive constants $a$ and $b$. 

  In the multi-type case, the matrix $\bm{A}$ plays a role analogous to that of $p_1$ in single type setting. Following \cite[Lemma 4.5]{JR2}, the asymptotic behavior of $\bm{A}$ can be classified into three distinct regimes:
\begin{itemize}
    \item  Schr\"{o}der Case: there exist $0<\gamma<1$ and positive integer $h$ such that $\gamma^{-n}\bm{A}^{nh}$ converges to a nonzero matrix $\bm{P}_0$ with finite entries;
    \item  B\"{o}ttcher Case: $\bm{A}$ is a nilpotent matrix, meaning $\bm{A}^j=\bm{0}$ for some positive integer $j$;
    \item  Intermediate Case: there exist $0<\hat{\gamma}<1$ and positive integer $\hat{h}$ such that $n^{-1}\hat{\gamma}^{-n}\bm{A}^{n\hat{h}}$ converges to a nonzero matrix with finite entries. 
\end{itemize}
We restrict our analysis to the first two cases (the Schr\"{o}der and B\"{o}ttcher cases), with the additional assumption $h=1$ in the Schröder case. In the following two subsections, we present the key properties and main results for each of these regimes separately.

\subsection{Schr\"{o}der case}
We begin by introducing the extreme lower deviation rate in the Schr\"{o}der case. From \cite[Theorem 1]{AV99}, 
\begin{equation}\label{en1}
	\bm{Q}_n(\bm{s}):=\cfrac{\bm{f}_n(\bm{s})}{\gamma^n}\rightarrow \bm{Q}(\bm{s}),~\bn{0}\le\bm{s}<\bm{1},
\end{equation}
component-wise $\text{as}~n\rightarrow \infty$. $\bm{Q}(\bm{s})$ is the unique solution of the  vector-valued functional equation $$
\bm{Q}(\bm{f}(\bm{s}))=\gamma\bm{Q}(\bm{s}) .$$
This implies that for any $i>0$ and $\bm{k}\in\mbb{Z}_+^d$,  $\lim_{n\to\infty}\gamma^{-n}\mbb{P}_i(\bm{Z}_n=\bm{k})$ exists and is positive.

Building on the framework of \cite{lower}, we aim to characterize extended lower deviation properties for multi-type GW processes. Specifically, we investigate the decay rate of $$\mbb{P}_i(\bm{Z}_n=\bm{k}_n),~~~\bm{k}_n\in\mbb{Z}^d~~\text{and}~~|\bm{k}_n|\le c_n.$$ 
 However, significant distinctions emerge compared to the single-type case. Recalling (\ref{cn}), the normalized population $\bm{Z}_n$ converges almost surely to the deterministic vector $W\bm{v}$. Consequently, we need to classify $\bm{k}_n$ according to its asymptotic relationship with $\bm{v}$. For $\bm{k}\in\mbb{Z}_+^d$, define its normalized distance with $\bm{v}$ as
\[
\varepsilon_v(\bm{k})=\left\Vert\frac{\bm{k}}{|\bm{k}|}-\bm{v}\right\Vert.
\]
For vectors $\bm{k}_1,\bm{k}_2\in\mbb{Z}_+^d$,  we define the partial order  $\bm{k}_1\succeq_{\bn{v}} \bm{k}_2$ if and only if
 $$|\bm{k}_1|\ge|\bm{k}_2|~~~\text{and}~~~\varepsilon_v(\bm{k}_1)\le\varepsilon_v(\bm{k}_2).$$ We say $\bm{k}\overset{\bn{v}}{\to}\infty$ if $$|\bm{k}|\to\infty ~~~\text{and}~~~\varepsilon_v(\bm{k})\to 0.$$ For ease of comparison, we adopt some notations of Fleischmann and Wachtel's theorems in the single-type framework. Our main result for the Schr\"{o}der case is stated as follows.
\begin{theorem}\label{thm1}
Define $a_k=a(\bm{k})=\min\{l\geq1:c_l\ge |\bm{k}|\}$ and $a'_k=a(k\bm{1})$. Then, for $1\le i \le d$,
	\begin{equation}\label{teq1}
	\sup_{\textbf{k}\hspace{0.2em}\succeq_{\textbf{v}}\hspace{0.2em}\tilde{\textbf{k}},~j\ge0			}\left| \frac{\rho^{j} c^d_{a_k}}{w_i(|\bm{k}|/(\rho^{j} c_{a_k}))} \mathbb{P}_i(\bm{Z}_{a_k+j} = \bm{k}) - 1 \right| \longrightarrow 0,~\text{as}~\tilde{\textbf{k}}\overset{\textbf{v}}{\longrightarrow}\infty
	\end{equation}
	and
\begin{equation}\label{teq2}
	\sup_{k\ge \tilde{k},~j\ge0	} \left| \frac{\mathbb{P}_i( |\bm{Z}_{a'_k+j}| \leq k)}{\mathbb{P}_i(0 < W < k/(\rho^{j} c_{a'_k}))} - 1 \right| \longrightarrow 0,~\text{as}~\tilde{k}\to\infty.
	\end{equation}
\end{theorem}
Clearly, this implies the following conclusion.
\begin{corollary}\label{cor1}
Suppose $|\bm{k}_n|\le c_n$ and $\bm{k}_n\overset{\textbf{v}}{\longrightarrow}\infty$ as $n\to\infty$. Then, for $1\le i \le d$,
	\[
\sup_{c_n\bn{1}~\succeq_{\textbf{v}}~\textbf{k}~\succeq_{\textbf{v}}~\textbf{k}_n			} \left| \frac{\rho^{n-a_k} c^d_{a_k}}{w_i(|\bm{k}|/(\rho^{n-a_k} c_{a_k}))} \mathbb{P}_i(\bm{Z}_{n} = \bm{k}) - 1 \right| \longrightarrow 0,~\text{as}~n\longrightarrow\infty
\]
and
\[
\sup_{c_n\bn{1}~\succeq_{\textbf{v}}~\textbf{k}~\succeq_{\textbf{v}}~\textbf{k}_n} \left| \frac{\mathbb{P}_i( |\bm{Z}_{n}| \leq k_n)}{\mathbb{P}_i(0 < W < k/(\rho^{n-a'_k} c_{a'_k}))} - 1 \right| \longrightarrow 0,\text{as}~n\longrightarrow\infty.
\]	
\end{corollary}
\begin{remark}
	The notation $a_k$ makes the above theorems and corollaries appear complex. We refer the reader to \cite[page 239]{lower} for detailed discussion on $a_k$. When $\sup_i\mbb{E}_i|\bm{Z}_1|\log |\bm{Z}_1|<\infty$ is satisfied, then $c_{a_k}=\rho^{a_k}$ and we have
	\[
	\mbb{P}_i(\bm{Z}_n=\bm{k})=\rho^{n-a_k}\rho^{da_k}w_i(|\bm{k}|/\rho^n)(1+\mathrm{o}(1)).\]
Compared to (\ref{ems2}), the equation above appears counterintuitive.  In fact, as will be demonstrated in the subsequent proof of the theorem, the key distinction from the single-type case is the decay rate for conditional probability:
\[
\mbb{P}(\bm{Z}_{a_k}=\bm{k}|\bm{Z}_0=\bm{l})=c_{a_k}^{-d}\bm{w}^{\ast\textbf{l}}(|\bm{k}|/c_{a_{k}})(1+\mathrm{o}(1)),\]
where $\bm{w}^{\ast\textbf{l}}(\cdot)=\prod_iw^{\ast l^{(i)}}_i(\cdot)$ (in the single-type case, the term $c_{a_k}^{-d}$  is replaced by $c_{a_k}^{-1}$). Therefore, together with
\[
\mbb{P}_i(\bm{Z}_n=\bm{k})=\sum_{\textbf{l}> 0}\mbb{P}(\bm{Z}_{n-a_k}=\bm{l})\mbb{P}(\bm{Z}_{a_k}=\bm{k}|\bm{Z}_0=\bm{l})\]
and
\begin{equation}
	\sum_{\textbf{l}> \bn{0}}\mbb{P}_i(\bm{Z}_j=\bm{l})\bm{w}^{\ast\textbf{l}}(k/c_{a_k})=\rho^{-j}w_i(|\bm{k}|/(c_{a_{k}}\rho^{j})),\end{equation}
roughly speaking, we can obtain Corollary \ref{cor1}.	
\end{remark}

In practical applications, we may need to estimate the lower deviation probability for each type of $\{\bm{Z}_n;n\ge0\}$. The following lemma addresses this issue. By an abuse of notation, let $a_k=\min\{l\geq1:c_l\ge k\}$.
\begin{theorem}\label{thm1.1}
 For $1\le i, m \le d$,
\begin{equation}\label{teq4}
	\sup_{k\ge\tilde{k},~j\ge0	}\left| \frac{\rho^{j} c_{a_k}}{w_i(k/(\rho^{j} c_{a_k}v^{(m)}))} \mathbb{P}_i(Z^{(m)}_{a_k+j} = k) - 1 \right| \longrightarrow 0,~\text{as}~\tilde{k}\to\infty
	\end{equation}
	and
	\begin{equation}\label{teq2n}
		\sup_{k\ge \tilde{k}, j\ge0	} \left| \frac{\mathbb{P}_i( Z^{(m)}_{a_k+j} \leq k)}{\mathbb{P}_i(0 < W < k/(\rho^{j} c_{a_k}v^{(m)}))} - 1 \right| \longrightarrow 0,~\text{as}~\tilde{k}\to\infty.
	\end{equation}
\end{theorem}

\subsection{B\"{o}ttcher case}
In this subsection, we present the properties and main results for a multi-type GW process in the B\"{o}ttcher case. \par At first, let us review some properties of the single-type GW process.  Recalling $u=\min\{i:p_i>0\}$, its offspring generating function $f$ is of B\"{o}ttcher type if $\mu\ge2$. A hallmark of this regime is the absence of extreme lower deviation events: specifically, 
\begin{equation}\label{kuz1}
\mbb{P}(Z_n<\mu^n)=0,~~\text{for all}~n\ge1.
\end{equation}	 Furthermore, from \cite[Theorem 6.9]{kuz}, the normalized offspring generating function sequence converges to a nontrivial limit:
\begin{equation}\label{kuz}
	\lim_{n\to\infty}(f_n(s))^{\mu^{-n}}>0,
\end{equation} where $f_n$ is the $n$-th iterate of $f$.\par
In the multi-type GW process framework, the generalization of the growth rate parameter 
$\mu$ requires careful construction. We begin by defining key combinatorial sets:\begin{itemize}
	\item Let~$J_{n,i}=\{\bm{j}\in\mbb{Z}_{+}^d:\mbb{P}_i(\bm{Z}_n=\bm{j})>0\}$, the set of accessible population vectors at generation 
	$n$ starting from type $i$.
	\item Let $\hat{J}_{n,i}$ be the lower boundary of $J_{n,i}$ with respect to the usual partial ordering on $\mbb{Z}_+^{d}$.
\end{itemize} 
For matrix operations, let $\bm{B}[k, \cdot]$ represent the $k$-th row vector of matrix $\bm{B}$. Define the minimal accessible matrix set:  $$K_n=\{\bm{B}\in\mbb{Z}_+^{d\times d }:\bm{B}[i,\cdot]\in\hat{J}_{n,i}, 1\le i \le  d\}$$ and the coordinate-wise minimal functional: $$\mathcal{K}_n^{(i)}(\bm{x})=\min_{\bn{s}\in\hat{J}_{n,i}}\bm{s}\cdot\bm{x},~~\bm{x}\in \mbb{R}^d.$$  Then the global minimization is then encapsulated as: $$
\mathcal{K}_n(\bm{x})=(\mathcal{K}_n^{(i)}(\bm{x}))_{1\le i \le d}=:\min_{\bn{B}\in K_n}\bm{B}\bm{x}.$$  
A critical recursive structure emerges from \cite[Lemma 1]{jones2}, for all $n>1$ and $\bm{x}\in\mbb{R}^d$,
\begin{equation}\label{e201}
	\mathcal{K}_n(\bm{x})=\mathcal{K}_1(\mathcal{K}_{n-1}\bm{x})=\min_{\bn{B}_1,\bn{B}_2,\cdots,\bn{B}_n\in K_1}\bm{B}_1\cdots \bm{B}_n\bm{x}.\end{equation} 
Finally, we define the multi-type growth rate parameter: $$\mu=\min_{\bn{B}\in K_1}\Vert \bm{B}\Vert,$$
where $\Vert \bm{B} \Vert$ denotes its spectral norm. By design, this definition aligns notationally with the single-type case which encodes the minimal asymptotic growth across interacting types. \par

Let us introduce the following condition:

\textbf{Condition (C1):}~There exists a unique (up to a scalar factor) $\bm{x}_0\gg\bm{0}$, a unique $\bm{B}_0\in K_1$ and $\mu_0\in(1,\rho)$ such that $\mathcal{K}(\bm{x}_0)=\bm{B}_0\bm{x}_0=\mu_0\bm{x}_0$. Furthermore, $\lim_{n\to\infty}\mu_0^{-n}\bm{B}_0^n$ exists.

Under Condition (\textbf{C1}), \cite[Theorem 2]{jones2} establishes two fundamental results: 
\begin{itemize}
	\item Growth rate identification: The minimal growth rate satisfies $\mu=\mu_0$.
	\item Logarithmic generating function limit: For any $\bm{s}\in(0,1]^d$, the scaled limit
\begin{equation}\label{e101}
	\bm{G}(\bm{s}):=\lim_{n\to\infty}-\mu^{-n}\log \bm{f}_n(\bm{x})
\end{equation} 
exists. Furthermore, the exponentiated limit $\overline{\bm{G}}(\bm{x}):=\exp \bm{G}(\bm{s})$ solves the multivariate B\"ottcher equation:
 $$\overline{\bm{G}}(\bm{f}(\bm{x}))=\overline{\bm{G}}(\bm{x})^{\mu\textbf{1}},$$
 where $\bm{1}$ is the all-ones vector.
 \end{itemize} 
 Define the minimal offspring vector at generation $n$: $$\hat{\bm{r}}_{n,i}:=\arg \min\{\bm{s}\in J_{n,i}:|\bm{s}|\}.$$ The parameter $\mu$ thus quantifies the minimal per-generation reproductive effort,  as formalized in the following proposition.
\begin{proposition}\label{auxlem8}
 For any $i$, there exist constants $0<\lambda_i<1<\lambda'_i$ such that, for $n\ge1$, \begin{equation}\label{jl1}\lambda_{i}\le \mu^{-n}|\hat{\bm{r}}_{n,i}| \le \lambda'_{i}.\end{equation}
	\end{proposition}
Actually, from condition (\textbf{C1}) and (\ref{e201}), $\mathcal{K}_n(\bm{x}_0)=\mu^n\bm{x}_0$ and thus $\min_{\bn{s}\in\hat{J}_{n,i}}\bm{s}\cdot\bm{x}_0=\mu^n x_0^{(i)}$ for $1\le i\le d$. Together with
	\[
	|\hat{\bm{r}}_{n,i}|\ge \left|\hat{\bm{r}}_{n,i}\cdot\frac{\bm{x}_0}{\max_k x_0^{(k)}}\right|\ge\frac{\min_{\bn{s}\in\hat{J}_{n,i}}\bm{s}\cdot\bm{x}_0}{\max_k x_0^{(k)}}\]
	and
	\[
	|\hat{\bm{r}}_{n,i}|\le \min_{\bn{s}\in\hat{J}_{n,i}}\bm{s}\cdot\frac{\bm{x}_0}{\min_k x_0^{(k)}}=\frac{\min_{\bn{s}\in\hat{J}_{n,i}}\bm{s}\cdot\bm{x}_0}{\min_k x_0^{(k)}}, \]
we get the desired proposition. \par
For $\bm{k}_n\in\mbb{Z}_+^d$, define \begin{equation}\label{delta}b_n=b(\bm{k}_n)=\min\{l\ge1 :c_l \mu^{n-l}\geq \lambda_{u}|\bm{k}_n|\},\end{equation}  where $\lambda_{u}=(\lambda_i\min_ku^{(k)})^{-1}$. Define $b'_n=b(k_n\bm{1})$ for $k_n\in \mbb{Z}_+$. Our main result in the B\"ottcher case is stated as follows: 
\begin{theorem}\label{thm2}
	Assume Condition (\textbf{C1}) holds. Let $\bm{k}_n\in\mbb{Z}_+^d$, $|\bm{k}_n|\ge|\hat{\bm{r}}_{n,i}|$, $|\bm{k}_n|=\mathrm{o}(c_n)$ and $\bm{k}_n\overset{\textbf{v}}{\to}\infty$ as $n\to\infty$. Then for $1\le i\le d$, 
	\begin{align}\label{ekl}
		-C_{i,1} &\leq \liminf_{n \to \infty} \mu^{b_n - n} \log \left[ c^d_n \,\mathbb{P}_i(\bm{Z}_n = \bm{k}_n) \right]\nonumber\\
		&\leq \limsup_{n \to \infty} \mu^{b_n - n} \log \left[ c^d_n\, \mathbb{P}_i(\bm{Z}_n = \bm{k}_n) \right] \leq -C_{i,2}.
	\end{align}
Assume $k_n\in\mbb{Z}_+$, $k_n\ge|\hat{\bm{r}}_{n,i}|$, $k_n=\mathrm{o}(c_n)$ and $k_n\to\infty$ as $n\to\infty$. Then for $1\le i\le d$,	
\begin{align}\label{ek2}
	-C_{i,3} &\leq \liminf_{n \to \infty} \mu^{b'_n - n} \log \left[  \mathbb{P}_i(|\bm{Z}_n| \le k_n) \right]\nonumber\\
	&\leq \limsup_{n \to \infty} \mu^{b'_n - n} \log \left[  \mathbb{P}_i(|\bm{Z}_n| \le k_n) \right] \leq -C_{i,4}.
\end{align}
\end{theorem}
Again, we give the estimate for lower deviation probability for each type of $\{\bm{Z}_n;n\ge0\}$ in the B\"ottcher case. By an abuse of notation, let $b_n=\min\{l\geq1:c_l \mu^{n-l}\geq \lambda_{u}k_n\}$.
\begin{theorem}\label{thm2.1}
	Assume Condition (\textbf{C1}) holds. Let $k_n\in\mbb{Z}_+$, $k_n\ge\min_{\bn{s}\in J_{n,i}}s^{(m)}$, $k_n=\mathrm{o}(c_n)$ and $k_n\to\infty$ as $n\to\infty$. Then for $1\le i,m\le d$, 
	\begin{align*}
		-C_{i,m,1} &\leq \liminf_{n \to \infty} \mu^{b_n - n} \log \left[ c_n \mathbb{P}_i(Z^{(m)}_n = k_n) \right]\nonumber\\
		&\leq \limsup_{n \to \infty} \mu^{b_n - n} \log \left[ c_n \mathbb{P}_i(Z^{(m)}_n = k_n) \right] \leq -C_{i,m,2},
	\end{align*}
	The above inequality remains true if $c_n\, \mathbb{P}_i(Z^{(m)}_n =k_n)$ is replaced by $ \mathbb{P}_i(Z^{(m)}_n \le k_n)$.
\end{theorem}
\section{Proofs in the Schr\"{o}der case}
The proofs of Theorems \ref{thm1} and \ref{thm2} are rooted in a multivariate adaptation of the classical Cramér method which is a well-known technique in large deviation theory originally designed for sums of independent random variables. Our approach extends this framework to multi-type branching processes, following the methodology introduced in \cite{lower}.  
\subsection{Cram\'er transformation for multi-type GW processes}
We begin by recalling the Cramér transform for scalar random variables and extend its definition to the multivariate setting. \par
 For a real-valued random variable $X$ with finite exponential moments  ($\mathbb{E}\mathrm{e}^{sX}<\infty$), its Cram\'er transformation $X(s)$ defined via the characteristic function: 
\[
\mathbb{E}\mathrm{e}^{\textup{i}tX(s)}=\frac{\mathbb{E}\mathrm{e}^{(s+\textup{i}t)X}}{\mathbb{E}\mathrm{e}^{sX}},~t\in\mbb{R}.\] 
For the $d$-dimensional GW process $\{\bm{Z}_n;n\ge0\}$ and any $\bm{h}\geq \bm{0}$, define its Cram\'er transform $\bm{X}_i(\bm{h},n),~1\le i \le d$ with density as
\begin{equation}\label{shn2}
\mbb{P}(\bm{X}_i(\bm{h},n)=\bm{k})=\frac{\mathrm{e}^{-\bn{h}\cdot\bn{k}/c_n}}{f_n^{(i)}(\mathrm{e}^{-\bn{h}/c_n})}\mbb{P}_i(\bm{Z}_n=\bm{k}).\end{equation}
For any $\bm{t}\geq \bm{0}$, its characteristic function is computed as: 
\begin{align}\label{ekm}
\mbb{E}(\mathrm{e}^{\textup{i}\bn{t}\cdot\bn{X}_i(\bn{h},n)})&=	\sum_{\bn{k}\geq\textbf{0}}\mbb{P}(\bm{X}_i(\bm{h},n)=\bm{k})\mathrm{e}^{\textup{i}\bn{t}\cdot\bn{k}}\nonumber\\
&=\sum_{\bn{k}\geq\textbf{0}}\frac{\mathrm{e}^{\textup{i}\bn{t}\cdot\bn{k}-\bn{h}\cdot\bn{k}/c_n}}{f_n^{(i)}(\mathrm{e}^{-\bn{h}/c_n})}\mbb{P}_i(\bm{Z}_n=\bm{k})
=\frac{f^{(i)}_n(\mathrm{e}^{\textup{i}\bn{t}-\bn{h}/c_n})}{f^{(i)}_n(\mathrm{e}^{-\bn{h}/c_n})} .
\end{align}
This demonstrates that  $\bm{X}_i(\bm{h},n)$ is the $d$-dimensional Cram\'er transformation of $\bm{Z}_n$, with parameter $-\bm{h}/c_n$, under the initial condition $\bm{Z}_0=\bm{e}_i$. 

Let $\{\bm{X}_{i,j}(\bm{h},n);j\geq1\}$ be i.i.d. copies of $\bm{X}_i(\bm{h},n)$.  Define
\[\bm{S}_{\bn{l}}(\bm{h},n)=\sum_{i=1}^d\sum_{j=1}^{l^{(i)}}\bm{X}_{i,j}(\bm{h},n),~~\bm{l}\in\mbb{Z}_+^d.\]
One can verify that
\begin{align}\label{eshn}
	\mbb{E}\exp\{\textup{i}\bm{t}\cdot\bm{S}_{\bn{l}}(\bm{h},n)\}=\frac{\left(\bm{f}_n(\mathrm{e}^{\textup{i}\bn{t}-\bn{h}/c_n})\right)^{\bn{l}}}{\left(\bm{f}_n(\mathrm{e}^{-\bn{h}/c_n})\right)^{\bn{l}}}
\end{align}
and
\begin{eqnarray}\label{shn1}
	\mbb{P}(\bm{Z}_n=\bm{k}|\bm{Z}_0=\bn{l})=\mathrm{e}^{\bn{h}\cdot\bn{k}/c_n}\left(\bm{f}_n(\mathrm{e}^{-\bn{h}/c_n})\right)^{\bn{l}}\mbb{P}(\bm{S}_{\bn{l}}(\bm{h},n)=\bm{k}).
\end{eqnarray}

\subsection{Basic estimates}
To estimate $\mbb{E}(\mathrm{e}^{\textup{i}\bn{t}\cdot\bn{X}_i(\bn{h},n)})$ via equation (\ref{ekm}), we analyze the asymptotic behavior of $f^{(i)}_n(\mathrm{e}^{\textup{i}\bn{t}-\bn{h}/c_n})$ across two distinct parameter regimes:
 \begin{align*}
 &J_{\epsilon}:=\{\bm{s}\in\mbb{R}^d:\vt\bm{s}\vt\in[\epsilon\pi ,\pi]\},~\epsilon\in(0,1)~~\text{and}\\
 &J_j:=\{\bm{t}\in\mbb{R}^d: \pi/c_j \le \vt\bm{t}\vt\le \pi/c_{j-1} \},~~j\geq1.
 \end{align*}
  \begin{lemma}[Estimate on $J_{\epsilon}$]\label{lem1}
  	For any $1\le i \le d$, there exists $\theta_i=\theta_i(\epsilon)\in(0,1)$ such that
  	\[f^{(i)}_k\left(\mathrm{e}^{-\textbf{h}/c_k+\mathrm{i} \textbf{t}/c_k}\right)\le \theta_i,~~k\geq0,~\bm{h}\geq\bm{0},~\bm{t}\in J_{\epsilon}.\]
  \end{lemma}
  \begin{proof}
  To analyze the characteristic functions, we define the complex exponential function: $$g_{\bn{h},\bn{t}}(\bm{x}):=\exp\{(-\bm{h}+\ii\bm{t})\cdot\bm{x}\},\bm{h},~\bm{x}\geq\bm{0},~ \bm{t}\in\mathbb{R}^d,$$ and establish its Lipschitz continuity. For any $\bm{x},\bm{y}\ge\bn{0}$,
  	\begin{align*}
  		\left\Vert g_{\bn{h}, \bn{t}}(\bn{x})-g_{\bn{h}, \bn{t}}(\bn{y})\right\Vert & =\left\Vert\mathrm{e}^{-\bn{h}\cdot\bn{x}}\left(\mathrm{e}^{\mathrm{i} \bn{t}\cdot\bn{x}}-\mathrm{e}^{\mathrm{i} \bn{t}\cdot\bn{y}}\right)+\mathrm{e}^{\mathrm{i} \bn{t}\cdot\bn{y}}\left(\mathrm{e}^{-\bn{h}\cdot\bn{x}}-\mathrm{e}^{-\bn{h}\cdot\bn{y}}\right)\right\Vert \\
  		& \leqslant\left\Vert\mathrm{e}^{\mathrm{i} \bn{t}\cdot\bn{x} }-\mathrm{e}^{\mathrm{i} \bn{t}\cdot \bn{y}}\right\Vert+\left\Vert\mathrm{e}^{-\bn{h}\cdot \bn{x}}-\mathrm{e}^{-\bn{h}\cdot \bn{y}}\right\Vert\\
  		& \leqslant(\Vert\bn{h}\Vert+\Vert\bn{t}\Vert)\Vert\bn{x}-\bn{y}\Vert,
  	\end{align*}
  	where the final inequality follows from the mean value theorem applied to the real and imaginary parts. This implies for any fixed $H<\infty$ and $T\geq d\pi$, the function family $$\mathcal{G}:=\{g_{\bn{h}, \bn{t}}(\cdot), \vt\bn{h}\vt\le H, \vt\bm{t}\vt\le T\}$$ 
  	is uniformly bounded and equi-continuous functions on $\mathbb{R}_{+}^d$. 	Therefore, by global limit theorem, for each $i$, 
  	\[
  	f_k^{(i)}\left(\mathrm{e}^{-\bn{h}/c_k+\mathrm{i} \bn{t}/c_k}\right)=\mbb{E}_ig_{\bn{h}, \bn{t}}(\bm{Z}_k/c_k)\rightarrow \mbb{E}_ig_{\bn{h}, \bn{t}}(\bm{v}W),~~\text{as} ~k\to\infty, \]
  	uniformly over $\mathcal{G}$. Since $W>0$ has an absolutely continuous distribution, it holds that 
  	\[
  	\sup_{0\le\vt\bn{h}\vt\le H, \bn{t}\in J_{\epsilon}}\sup_{1\le i\le d}\mbb{E}_ig_{\bn{h}, \bn{t}}(\bm{v}W)<1.\]
  Consequently, there exists $k_0$ and $\delta_1=\delta_1(k_0)\in(0,1)$,  
  	\[
  \sup_{0\le \vt\bn{h}\vt\le H, \bn{t}\in J_{\epsilon}}\sup_{1\le i\le d}f_k^{(i)}\left(\mathrm{e}^{-\bn{h}/c_k+\mathrm{i} \bn{t}/c_k}\right)\leq \delta_1,~k\geq k_0.	\]
For $k<k_0$, under Condition \textbf{(C1)} and the irreducibility of $\bm{M}$, the process is aperiodic and hence  
  \[
  \sup_{0\le \vt\bn{h}\vt\le H, \bn{t}\in J_{\epsilon}}\sup_{1\le i\le d}f_k^{(i)}\left(\mathrm{e}^{-\bn{h}/c_k+\mathrm{i} \bn{t}/c_k}\right)<1,~~k\le k_0.\]
  Combining the above two upper bounded yields the case $\vt\bn{h}\vt\le H$ in this lemma.\par 
 When $\vt\bn{h}\vt\ge H$, observing that for any $i,k$, we may choose $\bm{h}_0>\bm{0}$ with $\vt\bm{h}_0\vt\le H$ such that

  \[
 \left |f_k^{(i)}(\mathrm{e}^{-\bn{h}/c_k+\mathrm{i} \bn{t}/c_k})\right|\leq f_k^{(i)}(\mathrm{e}^{-\bn{h}/c_k})\le f_k^{(i)}(\mathrm{e}^{-\bn{h}_0/c_k}).
  \]
  As $k\to\infty$, the right-hand side converges to $\mbb{E}_i\mathrm{e}^{-W\bn{h}_0\cdot\bn{v}}$  which is strictly less than $1$ because $\bm{h}_0>\bm{0}$. Applying the same argument for $k< k_0$ and $k\ge k_0$ completes the proof.
  \end{proof}
 
  \begin{lem}[Estimate on $J_j$]\label{lem2}
  For any $1\le i \le d$, there exist constants $A_i>0$ and $A'_i\in(0,1)$ such that for $\bm{h}\geq0,\bm{t}\in J_{j}$, and $1\leq j\le n$,
  	\begin{eqnarray*}
  	\left|f^{(i)}_n\left(\mathrm{e}^{-\textbf{h}/c_n+\mathrm{i} \textbf{t}}\right)\right| \leqslant \begin{cases}A_i \gamma^{n-j+1} & \text { in the Schr\"{o}der case,} \\ (A'_i)^{\mu^{n-j+1}} & \text { in the B\"{o}ttcher case.}\end{cases}
  	\end{eqnarray*}
  \end{lem}
  \begin{proof}
  If $\bm{t}\in J_j, j\geq1$, then one can check that $
  	c_{j-1}\bm{t}\in J_{\epsilon}$, where  $\epsilon=\inf_{k\ge1}c_{k-1}/c_k\in(0,1)$ by our choice (\ref{cn1}).  Therefore, by Lemma \ref{lem1}, for $j\geq1,~\bm{h}\geq\bm{0}, \bm{t}\in J_j$,
  	\[
  |f_{j-1}^{(i)}(\mathrm{e}^{-\bn{h}/c_n+\mathrm{i}\bn{t}})|=|f_{j-1}^{(i)}(\mathrm{e}^{-\bn{h}(c_{j-1}/c_n)/c_{j-1}+\mathrm{i}(\bn{t}c_{j-1})/c_{j-1}})|\le \theta_i<1.\]
  	Using branching property, we recursively bound the generating function:
  	\[
  	|f_n^{(i)}(\mathrm{e}^{-\bn{h}/c_n+\mathrm{i}\bn{t}})|=\left|f_{n-j+1}^{(i)}\left(\bm{f}_{j-1}(\mathrm{e}^{-\bn{h}/c_n+\mathrm{i}\bn{t}})\right)\right|\le f_{n-j+1}^{(i)}(\theta\bm{1}),\]
 where $\theta=\max_{1\le i \le d}\theta_i$. \par 	In the Schr\"{o}der case, by (\ref{en1}), there exists positive constant $A_i$ such that $$|f_n^{(i)}(\mathrm{e}^{-\bn{h}/c_n+\mathrm{i}\bn{t}})|\le A_i\cdot\gamma^{n-j+1},~1\le i \le d.$$  \par 	 	
  	In the B\"{o}ttcher case, by (\ref{e101}), there exists positive constant $A'_i$ such that  $$|f_n^{(i)}(\mathrm{e}^{-\bn{h}/c_n+\mathrm{i}\bn{t}})|\le (A'_i)^{\mu^{n-j+1}}, ~1\le i \le d.$$ The proof is complete.
  \end{proof}
\subsection{Estimates on Le\'vy concentration functions}  
  For $r>0$, define the $d$-dimensional Le\'vy concentration function: $$Q(\bm{X};r)=\sup_{\bm{s}\in\mbb{R}^d}\mbb{P}(\vt\bm{X}-\bm{s}\vt\le r).$$
For ease of reference in subsequent proofs, we compile below a series of auxiliary lemmas governing the behavior of $Q(\bm{X};r)$. 
 \begin{lemma}\cite[Page 1]{Esseen}\label{auxlem1}
 	If $\bm{X}$ and $\bm{Y}$ are independent random vectors, then $Q(\bm{X}+\bm{Y};r)\le \min\{Q(\bm{X};r),Q(\bm{Y};r)\}$ for every $r>0$.
 \end{lemma}
 
\begin{lem}\label{auxlem2}{\cite[Lemma 6.1]{Esseen}}
  Let $\bm{X}$ be a $d$-dimensional random vector with the characteristic function 
  $$
  \varphi_{\textbf{X}}(\bm{t})=\mbb{E}(\mathrm{e}^{\textup{i}\textbf{t}\cdot\textbf{X}}),$$  and the concentration function $Q(\bm{X}; r)$. Then for parameters $a,r$ with $0\le ar\le 1$, the following estimate holds:
 	\begin{equation}\label{eqlem2}
 	Q(\bm{X}; r) \le C_d\cdot\int_{\vt\textbf{t}\vt\le a} |\varphi_{\textbf{X}}(\bm{t})|\mathrm{d}\bm{t},
 	\end{equation} 
 	where $C_d>0$ is a constant depending on the dimension $d$.
\end{lem}
\begin{lem}\label{auxlem3}{\cite[Theorem 4]{hazal}}
	Let $\bm{X}_1, \cdots, \bm{X}_n, \bm{X}_1',\cdots,\bm{X}_n'$ be independent $d$-dimensional random vectors such that $\bm{X}_i$ and $\bm{X}_i'$ have the same distributions for any $i$. Denote the distribution function of $\bm{X}_i^s:=\bm{X}_i-\bm{X}_i'$ by $F_i^s$. Then for $n\geq 8$,
	\begin{align*}
	Q(\sum_{i=1}^n\bm{X}_i;1)&\le C\cdot n^{-1}\left(\sum_{i=1}^nQ(\bm{X}_i^s;1)\right)D^{-d/2}
	\end{align*}
where $D=\inf_{\vt\bm{u}\vt_2=1}\sum_{i=1}^n\int_{\mbb{R}^d}\min\{1,(\bm{t}\cdot\bm{u})^2\}\mathrm{d}F^s_i(\bm{t}).$	
	\end{lem}
 By regarding $\bm{X}_{i,j}(\bm{h},n)$ as $\bm{X}_i$ in Lemma \ref{auxlem3}, actually we have
\[
Q\left(\sum_{i=1}^n\bm{X}_i;1\right)=\sup_{\bn{k}>\textbf{0}}\mbb{P}(\bm{S}_{\bn{l}}(\bm{h},n)=\bm{k}).\]
To estimate this concentration probability, we systematically apply the technical tools in lemmas \ref{lem1}-\ref{auxlem3}. 
\begin{lem}\label{lem3}
	For $\bm{h}\geq\bm{0}$ and $\bm{l}\in\mbb{Z}_+^d$ with $\gamma^{|\textbf{l}|}\rho^d<1$, there exists some constant $A(\bm{h})$ such that
	\begin{equation}\label{lemeq3}
	\sup_{n\geq1,\textbf{k}>\bn{0}}c^d_n\mbb{P}(\bm{S}_{\textbf{l}}(\bm{h},n)=\bm{k})\le A(\bm{h})|\bm{l}|^{-d/2}. \end{equation}
\end{lem}
\begin{proof}
{\bf Step 1}: Concentration Bound via Integral Inequality.

Let $\bm{l}_0\in \mbb{Z}_+^d$ satisfying $\gamma^{|\bn{l}_0|}\rho^d<1$. Applying Lemma \ref{auxlem2} with $\bm{X}=\bm{S}_{\bn{l}_0}(\bm{h},n)$, $r=1$ and $a=\pi$, we derive: 
	\begin{align*}
		\sup_{\bn{k}>\textbf{0}}\mbb{P}(\bm{S}_{\bn{l}_0}(\bm{h},n)=\bm{k})&\leq C\cdot\int_{\vt\bn{t}\vt\le \pi}|\mbb{E}\mathrm{e}^{\mathrm{i}\bn{t}\cdot \bn{S}_{\bn{l}_0}(\bn{h},n)}|\mathrm{d}\bm{t}\\
		&\le C\cdot\int_{\vt\bn{t}\vt\le \pi}\frac{\left|\bn{f}_n(\mathrm{e}^{\textup{i}\bn{t}-\bn{h}/c_n})^{\bn{l}_0}\right|}{\bn{f}_n(\mathrm{e}^{-\bn{h}/c_n})^{\bn{l}_0}}\mathrm{d}\bn{t},
	\end{align*}
	where we use (\ref{eshn}) in the last inequality. Since $\bm{Z}_n/c_n$ converges almost surely to a non-degenerate random variable, $\bm{f}_n(\mathrm{e}^{-\bn{h}/c_n})$ is uniformly bounded away from $\bm{0}$ (component-wise) for  $\bm{h}>\bm{0}$. When $\bm{h}=\bm{0}$, the denominator simplifies to $\bm{f}_n(\mathrm{e}^{-\bn{h}/c_n})=\bm{1}$. Therefore, 
	\begin{equation}\label{equ1}
\sup_{\bn{k}>\textbf{0}}\mbb{P}(S_{\bn{l}_0}(\bm{h},n)=\bm{k})\le C\cdot\int_{\vt\bn{t}\vt\le \pi}\left|\bn{f}_n(\mathrm{e}^{\textup{i}\bn{t}-\bn{h}/c_n})^{\bn{l}_0}\right|\mathrm{d}\bm{t}.	\end{equation}

{\bf Step 2}: Splitting the Integral Domain.

 Recall that
\[
J_j=\{\bm{t}\in\mbb{R}^d: \pi/c_j \le \vt\bm{t}\vt\le \pi/c_{j-1} \}.\]
The volume of $J_j$ can be calculated as
\[
|J_1|=\frac{\pi^{d/2}\pi^d}{\Gamma(1+d/2)}-\frac{\pi^{d/2}(\pi/c_1)^d}{\Gamma(1+d/2)},~|J_j|=\frac{\pi^{d/2}(\pi/c_{j-1})^d}{\Gamma(1+d/2)}-\frac{\pi^{d/2}(\pi/c_{j})^d}{\Gamma(1+d/2)},~j\ge2.\]
Hence by Lemma \ref{lem2}, we get for $j\geq1$
\begin{align*}
\int_{J_j}\left|\bn{f}_n(\mathrm{e}^{\textup{i}\bn{t}-\bn{h}/c_n})^{\bn{l}_0}\right|\mathrm{d}\bm{t}&\le C\cdot\gamma^{(n-j+1)|\bn{l}_0|}\left((\pi/c_{j-1})^d-(\pi/c_{j})^d\right)\\
&\le C\cdot\gamma^{(n-j+1)|\bn{l}_0|}c_{j-1}^{-d}.
\end{align*}
For the central region:
\[
\int_{\vt\bn{t}\vt\le \pi/c_n}\left|\bn{f}_n(\mathrm{e}^{\textup{i}\bn{t}-\bn{h}/c_n})^{\bn{l}_0}\right|d\bm{t}\le \frac{\pi^{d/2}(\pi/c_n)^d}{\Gamma(1+d/2)}.\]

{\bf Step 3}: Summation and Convergence.

Combining these estimates gives
\begin{align}\label{ekm2}
	c_n^d\int_{\vt\bn{t}\vt\le \pi}\left|\bn{f}_n(\mathrm{e}^{\textup{i}\bn{t}-\bn{h}/c_n})^{\bn{l}_0}\right|\mathrm{d}\bm{t}\le C\cdot(1+\sum_{j=1}^n\gamma^{(n-j+1)|\bn{l}_0|}(c_n/c_{j-1})^{d}).
\end{align}
Using the growth relation $c_n\le \rho^{n-j+1}c_{j-1}$, the series converges if $\gamma^{|\bn{l}_0|}\rho^d<1$,  which holds by the choice of $\bn{l}_0$. Thus, by (\ref{equ1}),
\begin{equation}\label{e1}
c_n^d\sup_{\bn{k}>\textbf{0}}\mbb{P}(S_{\bn{l}_0}(\bm{h},n)=\bm{k})\le A_0(\bm{h})
\end{equation}
for some constant $A_0(\bm{h})>0$. This establishes (\ref{lemeq3}) for $\bm{l}=\bm{l}_0$.

{\bf Step 4}: Generalization to $|\bm{l}|>|\bm{l}_0|$.

Let $\hat{a}:=\inf\{a\in \mbb{Z}_+:\gamma^a\rho<1\}$. By construction,  $\bm{l}_0=\hat{a}\bm{e}_i$ satisfies (\ref{e1}) for all $1\le i \le d$. For $\bm{l}$ with $|\bm{l}|>|\bm{l}_0|$,  decompose:
\[
\bm{l}=\sum_{i=1}^d\left\lfloor\cfrac{l^{(i)}}{\hat{a}}\right\rfloor\hat{a}\bm{e}_i+\bm{l}_{\text{rest}},
\] 
where $\lfloor\cdot\rfloor$ means the integer part and $\bm{l}_{\text{rest}}$ has components  $l^{(i)}-\lfloor l^{(i)}/\hat{a}\rfloor\hat{a}$. Applying Lemma \ref{auxlem1} with $r=1$ we get 
\begin{align*}
	Q(S_{\bn{l}}(\bm{h},n),1)&=\sup_{\bn{k}>\textbf{0}}\mbb{P}(S_{\bn{l}}(\bm{h},n)=\bm{k})\\&=\sup_{\bn{k}>\textbf{0}}\mbb{P}\left(\sum_{i=1}^d\sum_{j=1}^{\lfloor l^{(i)}/\hat{a}\rfloor}\bm{S}_{\hat{a}\bm{e}_i}(\bm{h},n,j)+\bm{S}_{\bn{l}_{\text{rest}}}(\bm{h},n)=\bm{k}\right)\\
	&\le\sup_{\bn{k}>\textbf{0}}\mbb{P}\left(\sum_{i=1}^d\sum_{j=1}^{\lfloor l^{(i)}/\hat{a}\rfloor}\bm{S}_{\hat{a}\bm{e}_i}(\bm{h},n,j)=\bm{k}\right),
\end{align*}
where $\bm{S}_{\hat{a}\bm{e}_i}(\bm{h},n,j), j\geq1$ are i.i.d. copies of $\bm{S}_{\hat{a}\bm{e}_i}(\bm{h},n)$. For $|\bm{l}|>8\hat{a}$, using Lemma \ref{auxlem1} and \ref{auxlem3} yields 
\begin{align}\label{e2}
	\sup_{\bn{k}>\textbf{0}}\mbb{P}\left(\sum_{i=1}^d\sum_{j=1}^{\lfloor l^{(i)}/\hat{a}\rfloor}\bm{S}_{\hat{a}\bm{e}_i}(\bm{h},n,j)=\bm{k}\right)\leq C\cdot |\bm{l}|^{-d/2} \max_{1\le i\le d}\{Q\left(\bm{S}_{\hat{a}\bm{e}_i}(\bm{h},n);1\right)\}.
\end{align}
From (\ref{e1}), we have the estimate for the concentration probability at $\bm{l}=\hat{a}\bm{e}_i$:
\[
c_n^dQ\left(\bm{S}_{\hat{a}\bm{e}_i}(\bm{h},n);1\right)\le C_i(\bm{h}),~1\le i \le d.\] Substituting into (\ref{e2}), there exists $A(\bm{h})>0$ such that for $\bm{l}\in\mbb{Z}_+^d$ with $\gamma^{|\bn{l}|}\rho^d<1$, 
\begin{equation}
	\sup_{n\geq1,\bn{k}>\textbf{0}}c_n^d\mbb{P}(\bm{S}_{\bn{l}}(\bm{h},n)=\bm{k})\le A(\bm{h})|\bm{l}|^{-d/2}.\end{equation}
This completes the proof.
\end{proof}
\begin{lem}\label{lem4}
For $\bm{l}\in\mbb{Z}_+^d$ with $\gamma^{|\textbf{l}|}\rho^d<1$, there exist positive constants $A$ and $\delta$ such that
\[
c_n^d\mbb{P}(\bm{Z}_n=\bm{k}|\bm{Z}_0=\bm{l})\le A|\bm{l}|^{-d/2}\mathrm{e}^{-\delta|\textbf{l}|+|\textbf{k}|/c_n}, n\geq1.\]
\end{lem}
\begin{proof}
Multiplying both side of equation (\ref{shn1}) by $c^d_n$, we obtain
\begin{equation}\label{emm1}
c^d_n\mbb{P}(\bm{Z}_n=\bm{k}|\bm{Z}_0=\bn{l})\le \mathrm{e}^{\bn{h}\cdot\bn{k}/c_n}\left(\bm{f}_n(\mathrm{e}^{-\bn{h}/c_n})\right)^{\bn{l}}\left(\max_{n\geq1,\bn{k}>\textbf{0}}c^d_n\mbb{P}(\bm{S}_{\bn{l}}(\bm{h},n)=\bm{k})\right).	
\end{equation}
Setting $\bm{h}=\bm{1}$ and applying Lemma \ref{lem3} to bound the concentration probability term, we get
\[
c^d_n\mbb{P}(\bm{Z}_n=\bm{k}|\bm{Z}_0=\bn{l})\le A(\bm{1})|\bm{l}|^{-d/2}\mathrm{e}^{|\bn{k}|/c_n}\left(\bm{f}_n(\mathrm{e}^{-1/c_n}\bm{1})\right)^{\bn{l}}.\]
By global limit theorem, for $1\le i \le d$, there exists $\delta_i>0$ such that $$\sup_nf^{(i)}_n(\mathrm{e}^{-1/c_n}\bm{1})<\mathrm{e}^{-\delta_i}.$$ Therefore,
\[
c^d_n\mbb{P}(\bm{Z}_n=\bm{k}|\bm{Z}_0=\bn{l})\le A(\bm{1})|\bm{l}|^{-d/2}\mathrm{e}^{|\bn{k}|/c_n}\mathrm{e}^{-\sum_{i=1}^d\delta_il^{(i)}}.\]
The Lemma thus follows by taking $\delta=\min_i\delta_i$. 
\end{proof}
\subsection{ A local limit theorem on $w(\cdot)$} 
Recall that for $\bm{k}\in\mbb{Z}_+^d$, its normalized distance with $\bm{v}$ is defined as
\[
\varepsilon_v(\bm{k})=\left\Vert\frac{\bm{k}}{|\bm{k}|}-\bm{v}\right\Vert.
\]
The following lemma demonstrates how the local limit theorem in the multi-type case differs from the laws (see \cite[Lemma 9]{local}) observed in the single-type case.
\begin{lem}\label{auxlem5}
Assume $\bm{x}_n\in\mbb{R}^d$ and $\sup_n|\bm{x}_n|<\infty$. For any $1\le i \le d $ and $\bm{l}\in\mbb{Z}_+^d$, if $\epsilon_v(\bm{x}_n)\to0~\text{as}~n\to\infty$,  we have
\begin{align*}
	\limsup_{n\to\infty}	\left|\frac{1}{(2\pi)^d}\int_{[-\pi c_n,\pi c_n]^d}\left(\bm{f}_n(\mathrm{e}^{\mathrm{i}\textbf{t}/c_n})\right)^{\textbf{l}}\mathrm{e}^{-\mathrm{i}\textbf{t}\cdot\textbf{x}_n}\mathrm{d}\bm{t}-\bm{w}^{\ast\textbf{l}}(|\bm{x}_n|)\right|=0,
\end{align*}
where $\bm{w}^{\ast\textbf{l}}(k):=w_{1}^{\ast l^{(1)}}(k)\ast w_{2}^{\ast l^{(2)}}(k)\cdots w_{d}^{\ast l^{(d)}}(k).$ If else, the integral decays faster than the convolution term:
\[
\int_{[-\pi c_n,\pi c_n]^d}\left(\bm{f}_n(\mathrm{e}^{\mathrm{i}\textbf{t}/c_n})\right)^{\textbf{l}}\mathrm{e}^{-\mathrm{i}\textbf{t}\cdot\textbf{x}}\mathrm{d}\bm{t}=\mathrm{o}(\bm{w}^{\ast\textbf{l}}(|\bm{x}_n|)).\]
\end{lem}
\begin{proof}
By a similar approach to the proof of \cite[Lemma 9]{local}, we can establish the uniform boundedness of the integral sequence:
\[
\sup_{n\geq0}	\left|\int_{[-\pi c_n,\pi c_n]^d}\left(\bm{f}_n(\mathrm{e}^{\mathrm{i}\bn{t}/c_n})\right)^{\bn{l}}\mathrm{e}^{-\mathrm{i}\bn{t}\cdot\bn{x}_n}\mathrm{d}\bm{t}\right|<\infty. \]
Therefore, from (\ref{cn}) and dominated convergence theorem, we obtain
\begin{equation}\label{ems3}
\limsup_{n\to\infty}	\left|\int_{[-\pi c_n,\pi c_n]^d}\left(\bm{f}_n(\mathrm{e}^{\mathrm{i}\bn{t}/c_n})\right)^{\bn{l}}\mathrm{e}^{-\mathrm{i}\bn{t}\cdot\bn{x}_n}\mathrm{d}\bm{t}-\int_{\mathbb{R}^d}\mathbb{E}(\mathrm{e}^{\mathrm{i}\bn{t}\cdot\bn{v}W}|\bm{Z}_0=\bm{l})\mathrm{e}^{-\mathrm{i}\bn{t}\cdot\bn{v}|\bn{x}_n|}\mathrm{d}\bm{t}\right|=0.\end{equation}
Define a vector function related to the subtracted term in the left-hand side of the above expression multiplying $(2\pi)^{-d}$: \begin{align}\label{skm2}
	T_0(\bm{x})&:=\frac{1}{(2\pi)^d}\int_{\mathbb{R}^d}\mathbb{E}(\mathrm{e}^{\mathrm{i}\bn{t}\cdot\bn{v}W}|\bm{Z}_0=\bm{l})\mathrm{e}^{-\mathrm{i}\bn{t}\cdot\bn{x}}\mathrm{d}\bm{t}\\
&=	\int_{\mathbb{R}^d}\mathbb{E}(\mathrm{e}^{2\pi\mathrm{i}\bn{t}\cdot\bn{v}W}|\bm{Z}_0=\bm{l})\mathrm{e}^{-2\pi\mathrm{i}\bn{t}\cdot\bn{x}}\mathrm{d}\bm{t}.\nonumber	\end{align}
Using $d$-dimensional inverse Fourier transform theorem yields
\begin{align*}
\mathbb{E}(\mathrm{e}^{2\pi\mathrm{i}\bn{t}\cdot\bn{v}W}|\bm{Z}_0&=\bm{l})=\int_{\mathbb{R}^d}T_0(\bm{x})\mathrm{e}^{2\pi\mathrm{i}\bn{t}\cdot\bn{x}}\mathrm{d}\bm{x}. 	\end{align*}
Thus, $T_0(\bm{x})$ coincides with the conditional probability density $T_0(\bm{x})=\mbb{P}(\bm{v}W= \bm{x}|\bm{Z}_0=\bn{l})$. By branching property, $T_0(\bm{x})$ is non-zero only when $\bm{x}$ aligns with the scaling direction:
\begin{align*}
	T_0(\bm{x})=\begin{cases}
		w_{1}^{\ast l^{(1)}}(k)\ast w_{2}^{\ast l^{(2)}}(k)\cdots w_{d}^{\ast l^{(d)}}(k),~&\text{if}~\bm{x}=k\bm{v}~\text{for some}~k>0,\\
		0, &\text{if else.}
	\end{cases}
\end{align*}
Combining (\ref{ems3}) we get the desired lemma.
\end{proof}
\begin{lem}\label{auxlem6}
	For any $\bm{l}$ satisfying $\gamma^{|\textbf{l}|}\rho^d<1$, there exists $A_0>0$ such that
	\[
\bm{w}^{\ast\textbf{l}}(x)\le A_0\prod_{i=1}^d\left(\int_{0}^{x}w_{i}(t)\mathrm{d}t\right)^{l^{(i)}},~x>0.	\]
\end{lem}
\begin{proof}
The proof follows a strategy analogous to \cite[Lemma 13]{lower}, with adaptations for the multi-type setting. We provide full details for completeness. Let $\boldsymbol{\varphi}=\boldsymbol{\varphi}_W$ denote the characteristic function of the limit variable $W$,  where
 $$\varphi_W^{(i)}(x)=\mathbb{E}_is^{\mathrm{i}xW},~~1\le i \le d.$$ By (\ref{cn}) and a continuity argument exactly analogous to the one-dimensional case, we have
\begin{align}\label{skm1}
	\varphi^{(i)}(\rho x)&=\lim_{n\to\infty}f_{n+1}^{(i)}(\mathrm{e}^{-\mathrm{i}\rho x/c_{n+1}}\bm{1})\nonumber\\
	&=f^{(i)}\left(\lim_{n\to\infty}f_{n}^{(i)}(\mathrm{e}^{-\mathrm{i} (\rho c_n/c_{n+1})x/c_{n}}\bm{1})\right)\nonumber\\
	&=f^{(i)}(\boldsymbol{\varphi}(x)).
\end{align}
Recalling that by (\ref{skm2}), for any $x>0$, the multi-type convolution density satisfies
\begin{align}
\bm{w}^{\ast\bn{l}}(x)&=\frac{1}{(2\pi)^d}\int_{\mathbb{R}^d}(\boldsymbol{\varphi}(\bm{v}\cdot\bm{t}))^{\bn{l}}\mathrm{e}^{-\mathrm{i}\bn{t}\cdot\bn{v}x}\mathrm{d}\bm{t}	\nonumber\\
&\le \frac{1}{(2\pi)^d}\int_{\mathbb{R}^d}\left|(\boldsymbol{\varphi}(\bm{v}\cdot\bm{t}))^{\bn{l}}\right|\mathrm{d}\bm{t}\label{esm3}\end{align}
Next, we prove the boundness of above equation.	Decompose $\mathbb{R}^d$ into
$
 \mathbb{R}^d=\bigcup_{k=0}^{\infty}I_k
 $, where
 $$
 I_0=\{\bm{x}\in \mathbb{R}^d:\vt\bm{x}\vt\le1\},~I_k=\{\bm{x}\in \mathbb{R}^d:\rho^{k-1}\le\vt\bm{x}\vt\le\rho^{k}\},~k\geq1.
 $$
 Therefore, for $k\ge1$, by (\ref{skm1}),
 \begin{align*}
\int_{I_{k+1}}\left|(\boldsymbol{\varphi}(\bm{v}\cdot\bm{t}))^{\bn{l}_0}\right|\mathrm{d}\bm{t}=\rho^{kd}\int_{I_1}\left|(\boldsymbol{\varphi}(\bm{v}\cdot\bm{t}\rho^k))^{\bn{l}}\right|\mathrm{d}\bm{t} 
=\rho^{kd}\int_{I_1}\left|(\bm{f}_k(\boldsymbol{\varphi}(\bm{v}\cdot\bm{t})))^{\bn{l}}\right|\mathrm{d}\bm{t}.
\end{align*}
Since $W>0$ has a positive continuous density, $|\varphi^{(i)}(\bm{v}\cdot\bm{t})|\le C<1$ uniformly for $1\le i\le d$ and $\bm{t}\in I_1$. By (\ref{en1}), this implies
\[
\int_{I_{k+1}}\left|(\boldsymbol{\varphi}(\bm{v}\cdot\bm{t}))^{\bn{l}_0}\right|\mathrm{d}\bm{t}\le C\cdot\rho^{kd}\gamma^{k|\bn{l}|}=C\cdot(\rho^d\gamma^{|\bn{l}|})^k.\]
Substituting this bound into (\ref{esm3}), we have $\sup_{x>0}\bm{w}^{\ast\bn{l}}(x)<\infty$ for any $\bm{l}\in\mbb{Z}_+^d$ satisfying $\rho^d\gamma^{|\bn{l}|}<1$. \par 
Finally, using the recursive structure of the convolution $w_i^{\ast l^{(i)}}(x)=\int_{0}^{x}w_i^{\ast (l^{(i)}-1)}(x-y)w_i(y)dy$, the lemma follows by induction on $\bm{l}$.
\end{proof}
\subsection{ Proof of Theorem \ref{thm1}} 

\textbf{Proof of equation (\ref{teq1}):}

Recall that $a_k=a(\bm{k})=\min\{l\geq1: c_l\ge |\bm{k}|\}$ and $a'_k=a(k\bm{1})$. By Markov property, we derive:
\[
\mbb{P}_i(\bm{Z}_{a_k+j}=\bn{k})=\sum_{\bn{l}>\textbf{0}}\mbb{P}_i(\bm{Z}_j=\bm{l})\mbb{P}(\bm{Z}_{a_k}=\bm{k}|\bm{Z}_0=\bm{l}).\]
and
\begin{equation}\label{em3}
\mbb{P}_i(|\bm{Z}_{a'_k+j}|\le k)=\sum_{\bn{l}>\textbf{0}}\mbb{P}_i(\bm{Z}_j=\bm{l})\mbb{P}(|\bm{Z}_{a'_k}|\le k|\bm{Z}_0=\bm{l}).
\end{equation}
Using the estimate in Lemma \ref{lem4}, for $N$ sufficiently large and $|\bm{l}|\geq N$, we have the bound: 
\begin{align*}
c^d_{a_k}\sum_{|\bn{l}|\geq N}\mbb{P}_i(\bm{Z}_j=\bm{l})\mbb{P}(\bm{Z}_{a_k}=\bm{k}|\bm{Z}_0=\bm{l})&\leq C\cdot\sum_{|\bn{l}|\geq N}|\bn{l}|^{-d/2}\mathrm{e}^{|\bn{k}|/c_{a_k}}\mathrm{e}^{-\delta|\bn{l}|}\mbb{P}_i(\bm{Z}_j=\bm{l})\\
&\le C\cdot N^{-d/2}\mathrm{e}^{|\bn{k}|/c_{a_k}}f^{(i)}_j(\mathrm{e}^{-\delta}\bm{1}).
\end{align*}
From the definition of $a_k$ and the assumption (\ref{cn1}) on $c_n$, we get the estimate:
\begin{equation}\label{em1}
1/\rho\le\frac{c_{a_k-1}}{c_{a_k}}\le\frac{|\bm{k}|}{c_{a_k}}\le1.\end{equation}
By the decay rate of $\bm{f}_n$ in (\ref{en1}), we have $f^{(i)}_j(\mathrm{e}^{-\delta}\bm{1})\le C \cdot\gamma^j$. Hence,
\begin{equation}\label{013}
\gamma^{-j}c^d_{a_k}\sum_{|\bn{l}|\geq N}\mbb{P}_i(\bm{Z}_j=\bm{l})\mbb{P}(\bm{Z}_{a_k}=\bm{k}|\bm{Z}_0=\bm{l})\le C\cdot N^{-d/2}.\end{equation}
Next, observe that the generating function has the following decomposition:
\[
\left(\bm{f}_n(\mathrm{e}^{2\pi\mathrm{i}\bn{t}})\right)^{\bn{l}}=\mbb{E}(\mathrm{e}^{2\pi\mathrm{i}\bn{t}\cdot\bn{Z}_n}|\bm{Z}_0=\bm{l})=\int_{\mbb{R}^d}\mbb{P}(\bm{Z}_n=\bm{k}|\bm{Z}_0=\bm{l})\mathrm{e}^{2\pi\mathrm{i}\bn{t}\cdot\bn{k}}\mathrm{d}\bm{k}.\]
Using $d$-dimensional inverse Fourier transform formula, we derive
\begin{align*}
\mbb{P}(\bm{Z}_n=\bm{k}|\bm{Z}_0=\bm{l})&=\frac{1}{(2\pi c_n)^d}\int_{[-c_n\pi,c_n\pi]^d}\left(\bm{f}_n(\mathrm{e}^{\mathrm{i}\bn{t}/c_n})\right)^{\bn{l}}\mathrm{e}^{-\mathrm{i}\bn{t}\cdot\bn{k}/c_n}\mathrm{d}\bm{t}.
\end{align*} 
Setting $n=a_k$, this becomes
\[
c_{a_k}^d\mbb{P}(\bm{Z}_{a_k}=\bm{k}|\bm{Z}_0=\bm{l})=\frac{1}{(2\pi )^d}\int_{[-c_{a_k}\pi,c_{a_k}\pi]^d}\left(\bm{f}_{a_k}(\mathrm{e}^{\mathrm{i}\bn{t}/c_{a_k}})\right)^{\bn{l}}\mathrm{e}^{-\mathrm{i}\bn{t}\cdot\bn{k}/c_{a_k}}\mathrm{d}\bm{t}.\]
Finally, by Lemma \ref{auxlem5}, for $\bm{k}\in\mbb{Z}_+^d$ and $\tilde{\bm{k}}\in\mbb{R}_+^d$,
\[
\sup_{\bn{k}\hspace{0.2em}\succeq_{\bn{v}}\hspace{0.2em}\tilde{\bn{k}}}\left|c_{a_k}^d\mbb{P}(\bm{Z}_{a_k}=\bm{k}|\bm{Z}_0=\bm{l})-\bm{w}^{\ast\bn{l}}(|\bm{k}|/c_{a_{k}})\right|\rightarrow 0,~\text{as}~\tilde{\bn{k}}\overset{\bn{v}}{\longrightarrow}\infty.\]
Therefore, using this relation we can estimate the first $N-1$ terms as follows:
\begin{equation}\label{011}
\sum_{|\bn{l}|\le N-1}\mbb{P}_i(\bm{Z}_j=\bm{l})\mbb{P}(\bm{Z}_{a_k}=\bm{k}|\bm{Z}_0=\bm{l})=c^{-d}_{a_{k}}\left[\sum_{|\bn{l}|\le N-1}\mbb{P}_i(\bm{Z}_j=\bm{l})\bm{w}^{\ast\bn{l}}(|\bm{k}|/c_{a_k})\right](1+\epsilon_{N,\bn{k}}),\end{equation}
where \begin{equation}\label{ekm3}\sup_{\bn{k}\hspace{0.2em}\succeq_{\bn{v}}\hspace{0.2em}\tilde{\bn{k}}}|\epsilon_{N,\bn{k}}|\to0,
~\text{as}~\tilde{\bn{k}}\overset{\bn{v}}{\longrightarrow}\infty.\end{equation} On the other hand, by Lemma \ref{auxlem6}, we bound the tail terms as follows:
\begin{align}
\sum_{|\bn{l}|\ge N}\mbb{P}_i(\bm{Z}_j=\bm{l})\bm{w}^{\ast\bn{l}}(|\bm{k}|/c_{a_k})&\le C\cdot\sum_{|\bn{l}|\ge N}\mbb{P}_i(\bm{Z}_j=\bm{l})\eta^{|\bn{l}|}\nonumber\\&\le C\cdot(\frac{\eta}{\eta_1})^Nf_j^{(i)}(\eta_1)\nonumber\\
&\le C\cdot \mathrm{e}^{-\delta N}\gamma^j,\label{012}
\end{align}
for some $\eta\in(0,1)$, $\delta>0$ and every $\eta_1\in(\eta,1)$. Combining (\ref{011}), (\ref{013}) and (\ref{012}), we derive
\begin{equation}\label{esm4}
\mbb{P}_i(\bm{Z}_{a_k+j}=\bn{k})=c^{-d}_{a_k}\left[\sum_{\bn{l}> \textbf{0}}\mbb{P}_i(\bm{Z}_j=\bm{l})\bm{w}^{\ast\bn{l}}(|\bm{k}|/c_{a_k})\right](1+\epsilon_{N,\bn{k}})+ \mathrm{O}(c^{-d}_{a_k}\gamma^{j}N^{-d/2}).\end{equation}
Using the functional equation $\boldsymbol{\varphi}(\rho^j x)=\bm{f}_j(\boldsymbol{\varphi}(x))$ and the decomposition of $\bm{f}_j$, we obtain
\begin{equation}\label{en3}
\sum_{\bn{l}> \textbf{0}}\mbb{P}_i(\bm{Z}_j=\bm{l})\bm{w}^{\ast\bn{l}}(|\bm{k}|/c_{a_k})=\rho^{-j}w_i(|\bm{k}|/(\rho^{j}c_{a_{k}})),~j\geq1.\end{equation}
From \cite[Corollary 6]{jones}, for any $1\le i \le d$, the density satisfies 
$$w_i(x)\sim C\cdot x^{\alpha-1},~\text{as}~x\to0,$$
where $\alpha=-\log\gamma/\log\rho$ is the Schr\"{o}der constant. Combining this with (\ref{em1}) yields
\[
\rho^{-j}w_i(|\bm{k}|/(\rho^{j}c_{a_{k}}))\ge C\cdot\rho^{-\alpha j}=C\cdot\gamma^j.\]
Thus, substituting into (\ref{esm4}), we get
\[\mbb{P}_i(\bm{Z}_{a_k+j}=\bn{k})=c^{-d}_{a_k}\rho^{-j}w_i(|\bm{k}|/(\rho^{j}c_{a_{k}}))(1+\epsilon_{N,\bn{k}}+\mathrm{O}(N^{-d/2})),\]
where $\epsilon_{N,\bn{k}}$ is defined in (\ref{ekm3}). Letting $\tilde{\bn{k}}\overset{\bn{v}}{\longrightarrow}\infty$ and then $N\to\infty$, we finish the proof.

\noindent\textbf{Proof of equation (\ref{teq2}):}

For independent positive random variable $X_1$ and $X_2$, it holds that $$\mbb{P}(X_1+X_2\le x)\le \mbb{P}(X_1\le x)\mbb{P}(X_2\le x).$$
Applying this inequality to the multi-type GW process, we derive
\[
\mbb{P}(|\bm{Z}_{a'_k}|\le k|\bm{Z}_0=\bm{l})\le \prod_{1\le i\le d}\left(\mbb{P}_i(|\bm{Z}_{a'_k}|\le k)\right)^{l^{(i)}}.\]
By global limit theorem (\ref{global}) and equation (\ref{em1}), for each type $i$,
\[
\mbb{P}_i(|\bm{Z}_{a'_k}|\le k)\le \mbb{P}_i(c_{a'_k}^{-1}|\bm{Z}_{a'_k}|\le 1)~\underset{k \to \infty}{\longrightarrow}~ \mbb{P}_i(0\le W \le 1).\]
Since $W>1$ has a positive density, there exists $\eta\in(0,1)$ such that for $k\geq N$,
\[
\mbb{P}(|\bm{Z}_{a'_k}|\le k|\bm{Z}_0=\bm{l})\le \eta^{|\bn{l}|}.\]
Following the approach in (\ref{012}), for $N$ sufficiently large, we have the bound:
\begin{equation}\label{em2}
\sum_{\bn{l}\ge N}\mbb{P}_i(\bm{Z}_j=\bm{l})\mbb{P}(|\bm{Z}_{a'_k}|\le k|\bm{Z}_0=\bm{l})\le C\cdot \mathrm{e}^{-\delta N}\gamma^j,~\delta>0.\end{equation}
\par Let $F^{\ast\bn{l}}(x)=\mbb{P}(W\le x|\bm{Z}_0=\bm{l})$ be the $\bm{l}$-fold convolution of $F(x)=\mbb{P}_i(W\le x)$. By global limit theorem and the continuity of $F$,
\[
\mbb{P}(|\bm{Z}_{a'_k}|\le c_{a'_k}x|\bm{Z}_0=\bm{l})\underset{k \to \infty}{\longrightarrow}F^{\ast\bn{l}}(x),~\text{uniformly in}~x\geq0. \]
Thus, 
\begin{equation}\label{em5}
\lim_{k\to\infty}\sup_{x\ge1}\left|\mbb{P}(|\bm{Z}_{a'_k}|\le x|\bm{Z}_0=\bm{l})-F^{\ast\bn{l}}(x/c_{a'_k})\right|=0.
\end{equation}
On the other hand, by the same argument to (\ref{em2}),
\begin{equation}\label{em4}
\sum_{\bn{l}\ge N}\mbb{P}_i(\bm{Z}_j=\bm{l})F^{\ast\bn{l}}(k/c_{a'_k})\le C\cdot \mathrm{e}^{-\delta N}\gamma^j.\end{equation}
Combining (\ref{em3}), (\ref{em2}), (\ref{em5}) and (\ref{em4}), we obtain
\begin{equation}\label{en2}
\mbb{P}_i(|\bm{Z}_{a'_k+j}|\le k)=\left[\sum_{\bn{l}>\textbf{0}}\mbb{P}_i(\bm{Z}_j=\bm{l})F^{\ast\bn{l}}(k/c_{a'_k})\right](1+\epsilon_{N,k})+
\mathrm{O}(\mathrm{e}^{-\delta N}\gamma^j),\end{equation}
where $\sup_{k\ge \tilde{k}}|\epsilon_{N,k}|\to0$ as $\tilde{k}\to\infty$ for fixed $N$. \par 
By (\ref{em1}), $k/c_{a_k}\geq\rho^{-1}$, ensuring $F^{\ast\bn{l}}(k/c_{a'_k})>0$ for all $\bm{l}=\bm{e}_i$. Together with the uniformly convergence for $\gamma^{n}\mbb{P}_i(\bm{Z}_n=\bm{e}_j)$ as $n\to\infty$, we derive  
\[
\mathrm{e}^{-\delta N}\gamma^j\le C\cdot \mathrm{e}^{-\delta N}\sum_{\bn{l}>\textbf{0}}\mbb{P}_i(\bm{Z}_j=\bm{l})F^{\ast\bn{l}}(k/c_{a'_k}).\]
Substituting into (\ref{en2}) gives
\[
\mbb{P}_i(|\bm{Z}_{a'_k+j}|\le k)=\left[\sum_{\bn{l}>\textbf{0}}\mbb{P}_i(\bm{Z}_j=\bm{l})F^{\ast\bn{l}}(k/c_{a'_k})\right](1+\epsilon_{N,k}+\mathrm{O}(\mathrm{e}^{-\delta N})).\]
Integrating both side of (\ref{en3}) yields
\[
\mbb{P}_i(W\le y/\rho^{j})=\sum_{\bn{l}>\textbf{0}}\mbb{P}_i(\bm{Z}_j=\bm{l})F^{\ast\bn{l}}(y),~y>0,~j\ge0.\]
Thus,
\[
\mbb{P}_i(|\bm{Z}_{a'_k+j}|\le k)=\mbb{P}_i(W\le \frac{k}{c_{a'_k}\rho^j})(1+\epsilon_{N,k}+\mathrm{O}(\mathrm{e}^{-\delta N})).\]
Letting $\tilde{k}\to\infty$ and then $ N\to\infty$, we obtain the desired result.
\subsection{ Proof of Theorem \ref{thm1.1}} \label{sec4}
The proof steps of Theorem \ref{thm1.1} are similar to those of Theorem \ref{thm1}, also relying on the method of Cram\'er transformation. Subsequently, I will introduce the necessary definitions, relations, and lemmas required for this proof. These conclusions all have counterparts in the proof of Theorem \ref{thm1}, and their proofs follow analogous reasoning to their corresponding versions. Therefore, we will omit the detailed proofs of these results. 
\par  For $1\le m \le d, n\ge1~\text{and}~ 0\le s\le 1$, define 
	\[
	f^{(i)}_{m,n}(s)=\mbb{E}_iZ_n^{(m)},~~\bm{f}_{m,n}(s)=(f^{(i)}_{m,n}(s))_{1\le i\le d}.\]
 One can prove that it satisfies:
	\[
	(\bm{f}_{m,n}(s))^{\bn{l}}=\mbb{E}(Z_n^{(m)}|\bm{Z}_0=\bm{l}),~\bm{l}\in\mbb{Z}_+^d,\]
	and
	\[
	\bm{f}_{m,n+1}(s)=\bm{f}\left(\bm{f}_{m,n}(s)\right), 0\le s\le 1,\]
	where $\bm{f}(\bm{s})$ is the offspring generating function for $\bm{Z}_1$ defined in \eqref{fs}.
\par 	For $h\ge0$, define the Cram\'er transformation $X_{i,m}(h,n)$ of $Z_n^{(m)}$ with density as:
	\begin{equation}\label{shn2n}
		\mbb{P}(X_{i,m}(h,n)=k)=\frac{\mathrm{e}^{-hk/c_n}}{f_{m,n}^{(i)}(\mathrm{e}^{-h/c_n})}\mbb{P}_i(Z^{(m)}_n=k).\end{equation}
Then, for any $t\ge0$, we have \begin{align}\label{ekmn}
		\mbb{E}(\mathrm{e}^{\textup{i}tX_{i,m}(h,n)})	
		=\frac{f_{m,n}^{(i)}(\mathrm{e}^{(\textup{i}t-h/c_n)})}{f_{m,n}^{(i)}(\mathrm{e}^{-h/c_n})}.
	\end{align}
Note that the global limit theorem (\ref{global}) implies that $$\lim_{n\to\infty}f_{m,n}^{(i)}(\mathrm{e}^{-h/c_n})=\mbb{E}_i\mathrm{e}^{-hv^{(m)}W},~1\le i,m \le d.$$
 Define \[S_{\bn{l},m}(h,n)=\sum_{i=1}^d\sum_{k=1}^{l^{(i)}}X_{i,m,k}(h,n),~~\bm{l}\in\mbb{Z}_+^d,\]
	where $\{X_{i,m,k}(h,n);k\geq1\}$ are i.i.d. copies of $X_{i,m}(h,n)$.
Then, we have \begin{align}\label{eshnn}
		\mbb{E}\mathrm{e}^{\textup{i}tS_{\bn{l},m}(h,n)}=\frac{\left(\bm{f}_{m,n}(\mathrm{e}^{-h/c_n+\textup{i}t})\right)^{\bn{l}}}{\left(\bm{f}_{m,n}(\mathrm{e}^{-h/c_n})\right)^{\bn{l}}}
	\end{align}
and  \begin{eqnarray}\label{shn1n}
		\mbb{P}(Z^{(m)}_n=k|\bm{Z}_0=\bn{l})=\mathrm{e}^{hk/c_n}\left(\bm{f}_{m,n}(\mathrm{e}^{-h/c_n})\right)^{\bn{l}}\mbb{P}(\bm{S}_{\bn{l},m}(h,n)=k).
	\end{eqnarray}
Define \begin{align*}
	&J'_{\epsilon}:=\{s\in\mbb{R}: s\in[\epsilon\pi ,\pi]\},~\epsilon\in(0,1)~~\text{and}\\
	&J'_j:=\{t\in\mbb{R}: \pi/c_j \le t\le \pi/c_{j-1} \},~~j\geq1.
\end{align*}
\begin{lemma}[Corresponds to Lemma \ref{lem1}]
	For any $1\le i,m \le d$, there exists $\theta'_i=\theta'_i(\epsilon)\in(0,1)$ such that
	\[f^{(i)}_{m,k}\left(\mathrm{e}^{-h/c_k+\mathrm{i} t/c_k}\right)\le \theta'_i,~~k\geq0,~h\ge 0,~t\in J'_{\epsilon}.\]
\end{lemma}
 \begin{lem}[Corresponds to Lemma \ref{lem2}]
	For any $1\le i,m \le d$, there exist constants $A'_i>0$ and $A_i'\in(0,1)$ such that for $h\geq0,t\in J'_{j}$, 
	\begin{eqnarray*}
		\left|f^{(i)}_{m,n}\left(\mathrm{e}^{-h/c_n+\mathrm{i} t}\right)\right| \leqslant \begin{cases}A'_i \gamma^{n-j+1} & \text { in the Schr\"{o}der case,} \\ (A'_i)^{\mu^{n-j+1}} & \text { in the B\"{o}ttcher case.}\end{cases}
	\end{eqnarray*}
\end{lem}
\begin{lem}[Corresponds to Lemma \ref{lem3}]
		For $h\ge0$ and $\bm{l}\in\mbb{Z}_+^d$ with $\gamma^{|\textbf{l}|}\rho<1$, there exists some constant $A'(h)$ such that
		\begin{equation}
			\sup_{n\geq1,\textbf{k}>\bn{0}}c_n\mbb{P}(S_{\textbf{l},m}(h,n)=k)\le A'(h)|\bm{l}|^{-d/2}. \end{equation}
\end{lem}
\begin{lem}[Corresponds to Lemma \ref{lem4}]
	For $\bm{l}$ with $|\bm{l}|$ sufficiently large, there exist positive constants $A'$ and $\delta'$ such that
	\[
	c_n\mbb{P}(Z^{(m)}_n=k|\bm{Z}_0=\bm{l})\le A'|\bm{l}|^{-d/2}\mathrm{e}^{-\delta'|\textbf{l}|+k/c_n}, n\geq1.\]
\end{lem}
\begin{lem}[Corresponds to Lemma \ref{auxlem5}]For $x>0$,
	\begin{align*}
		\limsup_{n\to\infty}	\left|\frac{1}{2\pi}\int_{[-\pi c_n,\pi c_n]}\left(\bm{f}_{m,a_k}(\mathrm{e}^{\mathrm{i}t/c_n})\right)^{\textbf{l}}\mathrm{e}^{-\mathrm{i}tx}dt-\bm{w}^{\ast\textbf{l}}(x/v^{(j)})\right|=0, 1\le j\le d.
	\end{align*}
\end{lem}

\noindent\textbf{Proof of Theorem \ref{thm1.1}:}

Using the following decomposition:
\[
\mbb{P}_i(Z^{(m)}_{a_k+j}=k)=\sum_{\bn{l}>\textbf{0}}\mbb{P}_i(\bm{Z}_j=\bm{l})\mbb{P}(Z^{(m)}_{a_k}=k|\bm{Z}_0=\bm{l}),\]
we can follow the same proof steps as in Theorem 1 to complete the proof of Theorem 2. The key difference lies in the decay rate of the conditional probability:
 \[
\lim_{\tilde{k}\to\infty} \sup_{k\ge\tilde{k}}\left|c_{a_k}\mbb{P}(Z^{(m)}_{a_k}=k|\bm{Z}_0=\bm{l})-\bm{w}^{\ast\bn{l}}\left(k/(c_{a_{k}}v^{(m)})\right)\right|=0\]
and
\begin{equation}\label{em5n}
	\lim_{k\to\infty}\sup_{x\ge1}\left|\mbb{P}(Z^{(m)}_{a_k}\le x|\bm{Z}_0=\bm{l})-F^{\ast\bn{l}}\left(x/(c_{a_k}v^{(m)})\right)\right|=0,~1\le m\le d.
\end{equation}

\section{Proofs in the B\"{o}ttcher case}
\subsection{ A vector-type local central limit theorem } 
The proof of Theorem \ref{thm2} is also based on a multivariate adaptation of the classical Cramér method. The proof involves intricate notation and we introduce them first. \par 

\begin{itemize}
	\item Recall that the density of $\bm{X}_i(\bm{h},n),~\bm{h}\ge\bm{0},$ is defined by
	\begin{equation}\label{shn3}
		\mbb{P}(\bm{X}_i(\bm{h},n)=\bm{k})=\frac{\mathrm{e}^{-\bn{h}\cdot\bn{k}/c_n}}{f_n^{(i)}(\mathrm{e}^{-\bn{h}/c_n})}\mbb{P}_i(\bm{Z}_n=\bm{k}).\end{equation}
	\item Define its normalization by  $\overline{\bm{X}}_{i,j}(\bm{h},n)=\mbb{E}\bm{X}_{i,j}(\bm{h},n)-\bm{X}_{i,j}(\bm{h},n)$ and the partial sum by
	\[\overline{\bm{S}}_{\bn{l}}(\bm{h},n)=\sum_{i=1}^d\sum_{j=1}^{l^{(i)}}\overline{\bm{X}}_{i,j}(\bm{h},n),~\bm{l}\in\mbb{Z}_+^d,\]
	where $\bm{X}_{i,j}(\bm{h},n),~j\ge1$, are i.i.d. copies of $\bm{X}_{i}(\bm{h},n)$.
	\item The third moment of $\overline{\bm{S}}_{\bn{l}}(\bm{h},n)$ is defined by\[
	\hat{\rho}_3=	\hat{\rho}_3(\bm{h},n)=|\bm{l}|^{-1}\left(\sum_{i=1}^d\sum_{j=1}^{l^{(i)}}\mbb{E}\Vert\overline{\bm{X}}_{i,j}(\bm{h},n)\Vert^3 \right).\]
	\item The third order Liapounov coefficient of $\overline{\bm{S}}_{\bn{l}}(\bm{h},n)$ is defined by
	\[
	L_{3,\bn{l}} = L^{\bn{h},n}_{3,\bn{l}}=\sup_{\|\bn{t}\|=1} \frac{|\bm{l}|^{-1}\sum_{i=1}^d\sum_{j=1}^{l^{(i)}} \mbb{E}\left(|\langle \bm{t}, \overline{\bm{X}}_{i,j}(\bm{h},n) \rangle|^3\right)}{\left[|\bm{l}|^{-1} \sum_{i=1}^d\sum_{j=1}^{l^{(i)}} \mbb{E}\left(\langle \bm{t}, \overline{\bm{X}}_{i,j}(\bm{h},n) \rangle^2\right)\right]^{3/2}} \cdot |\bm{l}|^{-\frac{1}{2}}.
	\]
	\item The average covariance matrix of $\overline{\bm{X}}_{i,j}(\bm{h},n)$ is defined by$$
	\bm{V}=\bm{V}^{\bn{h},n}=|\bm{l}|^{-1}\sum_{i=1}^d\sum_{j=1}^{l^{(i)}}\text{Cov}(\overline{\bm{X}}_{i,j}(\bm{h},n)), ~\bm{h}\ge\bm{0},~n\geq1.
	$$
\item	Since the GW process $\{\bm{Z}_n;n\ge0\}$ is non-singular and aperiodic,  the matrix $\bm{V}^{\bn{h},n}$ is invertible. Let $\bm{B}=\bm{B}^{\bn{h},n}$ be a symmetric and positive-definite matrix satisfying  $\bm{B}^2=(\bm{V}^{\bn{h},n})^{-1}.$ 
\item  The normalized characteristic function of $\overline{\bm{S}}_{\bn{l}}(\bm{h},n)$ is defined by
\[
\psi_{\bn{l}}(\bm{t})=\psi^{\bn{h},n}_{\bn{l}}(\bm{t})=\mathbb{E}\exp\{\mathrm{i}\vert\bm{l}\vert^{-1/2}(\bm{B}\bm{t})\cdot\overline{\bm{S}}_{\bn{l}}(\bm{h},n)\},~\bm{t}\in\mbb{R}^d.\]
\end{itemize}

The following lemma establishes a multivariate analogue of the Essen inequality, generalizing \cite[Lemma 14]{lower} to the multi-type framework.   
\begin{lemma}\label{auxlem7}
For fixed bounds $0<h_1\le h_2 <\infty$ and $\bm{l}\in\mbb{Z}_+^d$, there exist constants $C_k:=C_k(h_1,h_2)>0,k=1,2,3$, such that for all $\Vert\bm{t}\Vert\le C_1\vert \bm{l} \vert^{1/2}$, the following uniform bound holds:
\begin{align}\label{e301}
\sup_{\Vert\textbf{h}\Vert\in(h_1,h_2),n\geq1}\left|\psi_{\textbf{l}}(\bm{t})-\exp\{-\Vert \bm{t}\Vert^2/2\}\right|\le C_2\vert \bm{l}\vert^{-1/2}\Vert \bm{t}\Vert^3\mathrm{e}^{-C_3\Vert \textbf{t}\Vert^2}.
\end{align}
\end{lemma}
\begin{proof}
To establish the lemma, we first prove that the Lyapunov coefficient  $L_{3,\bn{l}}$ of the normalized sum $\overline{\bm{S}}_{\bn{l}}(\bm{h},n)$ satisfies the uniform bound:
 \begin{equation}\label{eq11}
 \sup_{\Vert\bn{h}\Vert\in[h_1,h_2],n\geq1}L_{3,\bn{l}}\le C\cdot|\bm{l}|^{-1/2}.
 	\end{equation}
The desired result for $\Vert \bm{t} \Vert\le C_2\vert \bm{l} \vert^{1/6}$ then follows directly from an application of the central limit theorem for vectors [Theorem 8.4, \cite{gauss}] to the sequence $\overline{\bm{X}}_{i,j}(\bm{h},n)$. \par To this end, by equation (8.12) of \cite{gauss}, we get the upper bound for $L_{3,\bn{l}}$:
 \[
 L_{3,\bn{l}}\le\hat{\rho}_3 \hat{\lambda}^{-3/2} |\bm{l}|^{-1/2},\] 
 where $\hat{\lambda}$ denote the least singular value of the average covariance matrix $\bm{V}^{\bn{h},n}$.
 From (\ref{shn2}) and the definition of $\bm{V}^{\bn{h},n}$, the scaled matrix $c_n^{-2}\bm{V}^{\bn{h},n}$ converges to a positive-definite matrix as $n\to\infty$ if $\Vert\bm{h}\Vert\in[h_1,h_2]$. By singular value decomposition, this implies $$\lim_{n\to\infty}c_n^{-2}\hat{\lambda}\in(0,\infty). $$   Furthermore, since $\lim_{n\to\infty}c_n^{-3}\mbb{E}\Vert\overline{\bm{X}}_{i,j}(\bm{h},n)\Vert^3\in(0,\infty)$, we have the boundedness of $c_n^{-3}\hat{\rho}_3$. Combining these results, we verify (\ref{eq11}). \par 
However, since there exist only finitely many distinct $\overline{\bm{X}}_{i,j}(\bm{h},n)$ in the sum of $\overline{\bm{S}}_{\bn{l}}(\bm{h},n)$ for any $\bm{l}>\bm{0}$, we extend the validity to $\Vert \bm{t} \Vert\le C_2\vert \bm{l} \vert^{1/2}$ according the proof of \cite[Theorem 8.4]{gauss}(or see \cite[equation(8.23)]{gauss}). The proof is complete.
\end{proof}
\begin{lemma}[A vector-type local central limit theorem]\label{auxlem11}
	For $0< h_1\le h_2<\infty$,
	\[
	\sup_{\Vert \textbf{h}\Vert\in[h_1,h_2],n\geq1}\sup_{\textbf{k}\in\mbb{Z}_+^d}\left|\vert\textbf{l}\vert^{d/2}|\bm{B}|^{-1}\mbb{P}(\textbf{S}_{\textbf{l}}(\textbf{h},n)=\bm{k})-(2\pi)^{-d/2}\mathrm{e}^{-\vt\textbf{x}_{\textbf{l},\textbf{k}}\vt^2/2}\right|\xrightarrow[|\textbf{l}|\to\infty]{} 0,
	\]
where $\textbf{x}_{\textbf{l},\textbf{k}}:=\textbf{x}_{\textbf{l},\textbf{k}}(\textbf{h},n):=|\bm{l}|^{-1/2}(\textbf{k}-\mbb{E}\bm{S}_{\textbf{l}}(\textbf{h},n))\bm{B}$.
\end{lemma}
\begin{proof}
	By $d$-dimensional inverse Fourier transform theorem,
	\[
	\mbb{P}(\bm{S}_{\bn{l}}(\bm{h},n)=\bm{k})=\frac{1}{(2\pi )^d}\int_{[-\pi,\pi]^d}\mathrm{e}^{-\mathrm{i}\bn{t}\cdot\bn{k}}\mbb{E}\mathrm{e}^{\mathrm{i}\bn{t}\cdot\bn{S}_{\bn{l}}(\bn{h},n)}\mathrm{d}\bm{t}.\]
Change the variable $\bm{t}$ to $\vert \bm{l}\vert^{-1/2}\bm{B}\bm{t}^T$,	we get
\begin{eqnarray*}
\mbb{P}(\bm{S}_{\bn{l}}(\bm{h},n)=\bm{k})=\frac{|\bm{B}|}{(2\pi\vert \bm{l}\vert^{1/2})^d}\int_{\Omega} \exp\{-\mathrm{i}\vert \bm{l}\vert^{-1/2}(\bm{B}\bn{t})\cdot(\bn{k}-\mbb{E}\bm{S}_{\bn{l}}(\bm{h},n))\}\psi_{\bn{l}}(\bm{t})\mathrm{d}\bm{t},
\end{eqnarray*}
where $|\bm{B}|$ is determinant of $\bm{B}$ and $\Omega=\{\bm{t}\in\mbb{R}^d:\vert \bm{l}\vert^{-1/2}\bm{B}\bm{t}^T\in[-\pi,\pi]^d\}$. From Lemma \ref{auxlem7}, there is a nice connection between $\psi_{\bn{l}}(\bm{t})$ and $\mathrm{e}^{-\vt\bn{t}\vt/2}$. On the other hand, we have the following identity:
\begin{align*}
\int_{\mbb{R}^d}\exp\{-\mathrm{i}\bm{x}_{\bn{l},\bn{k}}\cdot\bm{t}-\Vert \bm{t}\Vert^2/2\}d\bm{t}
=&\int_{\mbb{R}^d}\exp\left\{-\frac{1}{2}[(\bm{t}+\mathrm{i}\bm{x}_{\bn{l},\bn{k}})^2+(\bm{x}_{\bn{l},\bn{k}})^2]\right\}\mathrm{d}\bm{t}\\
=&(2\pi)^{d/2}\mathrm{e}^{-(\bm{x}_{\bn{l},\bn{k}})^2/2}.
\end{align*}
Therefore, for any $0<\epsilon<A<\infty$, we have
\[
\sup_{\bn{k}>\textbf{0}}\left|\vert\bm{l}\vert^{d/2}|\bm{B}|^{-1}\mbb{P}(\bm{S}_{\bn{l}}(\bm{h},n)=\bm{k})-(2\pi)^{-d/2}\mathrm{e}^{-\bm{x}_{\bn{l},\bn{k}}(\bn{h},n)^2/2}\right|\le C\cdot(I_1+I_2+I_3+I_4),\]
where
\begin{align*}
&I_1:=\int_{\Omega_1}|\psi_{\bn{l}}(\bm{t})-\mathrm{e}^{-\bn{x}_{\bn{l},\bn{k}}(\bn{h},n)^2/2}|\mathrm{d}\bm{t}, ~~~~I_2:=\int_{\Omega_2}\mathrm{e}^{-\bn{x}_{\bn{l},\bn{k}}(\bn{h},n)^2/2}\mathrm{d}\bm{t},\\
&I_3:=\int_{\Omega_3}|\psi_{\bn{l}}(\bm{t})|\mathrm{d}\bm{t},~~~~~~~~~~~~~~~~~~~~~~~I_4:=\int_{\Omega_4}|\psi_{\bn{l}}(\bm{t})|\mathrm{d}\bm{t},	
	\end{align*}
and
\begin{align*}
&	\Omega_1=\{\bm{t}\in\mbb{R}^d:\vert \bm{l}\vert^{-1/2}\bm{B}\bm{t}^T\in[-\epsilon,\epsilon]^d\},~~~~~\Omega_2=\{\bm{t}\in\mbb{R}^d:\vert \bm{l}\vert^{-1/2}\bm{B}\bm{t}^T\in[-\epsilon,\infty)^d\}\\
&	\Omega_3=\{\bm{t}\in\mbb{R}^d:\vert \bm{l}\vert^{-1/2}\bm{B}\bm{t}^T\in[-\epsilon,A]^d\},~~~~~\Omega_4=\{\bm{t}\in\mbb{R}^d:\vert \bm{l}\vert^{-1/2}\bm{B}\bm{t}^T\in[A,\pi]^d\}.
\end{align*}

With the foundational results established in Lemmas \ref{lem2} and \ref{auxlem7}, along with the auxiliary relations (\ref{cn}) and (\ref{e101}), one can prove that $I_i,1\le i\le 4$, asymptotic decay to zero as $|\bm{l}|\to\infty$ by the same methodology presented in \cite[Lemma 15]{lower}. To avoid redundancy, we omit the technical details here.
\end{proof}
\subsection{Existence of Bounded Matching Parameters}
 In applying Lemma \ref{auxlem11}, our objective is to find suitable $\bm{h}_n$ such that $\lim_{n\to\infty}\bm{x}_{\bn{l}_n,\bn{k}_n}(\bm{h}_n,n) = 0$ for given $\bm{l}_n,\bm{k}_n\in\mbb{Z}_+^d$.  The following lemma shows this would happen if $\bm{k}_n\overset{\bn{v}}{\longrightarrow}\infty$ as $n\to\infty$. Recall that $$b_n=\min\{n:c_l \mu^{n-l}\geq \lambda_{u}|\bm{k}_n|\},~J_{n,i}=\{\bm{j}\in\mbb{Z}_{+}^d:\mbb{P}_i(\bm{Z}_n=\bm{j})>0\}$$  $\hat{J}_{n,i}$ is the lower boundary of $J_{n,i}$ and $\lambda_{u}$ is the constant defined in (\ref{delta}).  Recall the minimal offspring vector at generation $n$ is defined as: $$\hat{\bm{r}}_{n,i}=\arg \min\{\bm{s}\in J_{n,i}:|\bm{s}|\}.$$ 
\begin{lemma}\label{auxlem12}
	Assume $\bm{k}_n\in\mbb{Z}_+^d$ and $\bm{k}_n\overset{\textbf{v}}{\longrightarrow}\infty$ as $n\to\infty$. There exists a sequence of $\bm{h}_n$ with $0<\liminf_{n\to\infty}\vt\bm{h}_n\vt \le  \limsup_{n\to\infty}\vt\bm{h}_n\vt<\infty$, such that $\lim_{n\to\infty}|\mbb{E}\bm{S}_{\hat{\textbf{r}}_{n-b_n,i}}(\bm{h}_n,b_n)-\bm{k}_n|=0$.
\end{lemma}
\begin{proof}
By definition of $\bm{S}_{\bn{l}}(\bm{h},n)$ and $\bm{X}_{i}(\bm{h},n)$, we derive the first-moment relationship:	$$\mbb{E}\bm{S}_{\hat{\bn{r}}_{n-b_n,i}}(\bm{h}_n,b_n)=\hat{\bm{r}}_{n-b_n,i}\bm{M}(\bm{h}_n,b_n),~1\le i \le d,~n\geq1,$$ where $\bm{M}(\bm{h}_n,b_n)$ denotes the $d\times d $ matrix with entries $M_{i,j}=\mbb{E}X^{(j)}_{i}(\bm{h}_n,b_n)$~(recall that $X^{(j)}$ denotes the $j$-th component of $\bm{X}$).\par
 Define the normalized operator:
 \[
\M_n(\bm{h})=\hat{\bm{r}}_{n-b_n,i}\bm{M}(\bm{h},b_n) =c_{b_n}\hat{\bm{r}}_{n-b_n,i}\frac{\bm{M}(\bm{h},b_n)}{c_{b_n}}.\]
The scaled matrix entries of $\bm{M}(\bm{h},b_n)/c_{b_n}$ satisfy
\begin{equation}
\frac{M_{i,j}}{c_{b_n}}=\frac{\mbb{E}X^{(j)}_{i}(\bm{h}_n,b_n)}{c_{b_n}}=\cfrac{\frac{\partial f_{b_n}^{(i)}}{\partial s_j}(\mathrm{e}^{-\bn{h}/c_{b_n}})\mathrm{e}^{-h^{(j)}/c_{b_n}}}{c_{b_n}f_{b_n}^{(i)}(\mathrm{e}^{-\bn{h}/c_{b_n}})},\end{equation}
where $\frac{\partial f_{b_n}^{(i)}}{\partial s_j}(\cdot)$ denotes the partial derivative of $f_{b_n}^{(i)}$ with respect to the $j$-th component. \par 
Next,  we look at asymptotic behaviors of $\M_n(\bm{h})$. By the convergence of $\bm{Z}_n/c_n$ (see (\ref{cn})), we obtain
\begin{equation}\label{emm2}
\lim_{n\to\infty}\frac{M_{i,j}}{c_{b_n}}= \cfrac{-\frac{\partial \mbb{E}_{i}\mathrm{e}^{-\bn{h}\cdot\bn{v}W}}{\partial h^{(j)}}}{\mbb{E}_{i}\mathrm{e}^{-\bn{h}\cdot\bn{v}W}}=\cfrac{v^{(j)}\mbb{E}_{i}W\mathrm{e}^{-\bn{h}\cdot\bn{v}W}}{\mbb{E}_{i}\mathrm{e}^{-\bn{h}\cdot\bn{v}W}}.\end{equation}
Thus, for any $\epsilon>0$, there exists $n_1>0$ such that for $n\geq n_1$,
\[
(1-\epsilon)c_{b_n}\left(\sum_{k=1}^d\hat{r}_{n-b_n,i}^{(k)}\cfrac{\mbb{E}_{k}W\mathrm{e}^{-\bn{h}\cdot\bn{v}W}}{\mbb{E}_{k}\mathrm{e}^{-\bn{h}\cdot\bn{v}W}}\right)\cdot\bm{v}\le\M_n(\bm{h})\le (1+\epsilon)c_{b_n}\left(\sum_{k=1}^d\hat{r}_{n-b_n,i}^{(k)}\cfrac{\mbb{E}_{k}W\mathrm{e}^{-\bn{h}\cdot\bn{v}W}}{\mbb{E}_{k}\mathrm{e}^{-\bn{h}\cdot\bn{v}W}}\right)\cdot\bm{v}\]
and
\[
(1-\epsilon)|\bm{k}_n|\cdot\bm{v}\le\bm{k}_n\le (1+\epsilon)|\bm{k}_n|\cdot\bm{v}.\]
We aim to show that there exists a sequence of $\bm{h}_n,n\ge n_1,$ with $0<\liminf_{n\to\infty}\vt\bm{h}_n\vt \le  \limsup_{n\to\infty}\vt\bm{h}_n\vt<\infty$ such that
\begin{equation}\label{key}
c_{b_n}\left(\sum_{k=1}^d\hat{r}_{n-b_n,i}^{(k)}\cfrac{\mbb{E}_{k}W\mathrm{e}^{-\bn{h}_n\cdot\bn{v}W}}{\mbb{E}_{k}\mathrm{e}^{-\bn{h}_n\cdot\bn{v}W}}\right)=|\bm{k}_n|.\end{equation}
If this holds, then for $n\geq n_1$,\[
-2\epsilon\le |\M_n(\bm{h}_n)-\bm{k}_n|\le 2\epsilon\]
and the lemma follows by letting $\epsilon \to 0$.\par 
For each $1\le k\le d$, we observe the asymptotic behavior of the ratio at boundaries:
\begin{equation}\label{eqlem11}
\frac{\mbb{E}_{k}W\mathrm{e}^{-\bn{h}\cdot\bn{v}W}}{\mbb{E}_{k}\mathrm{e}^{-\bn{h}\cdot\bn{v}W}}\bigg\vert_{\bn{h}=\textbf{0}}=\mbb{E}_kW~~~\text{and}~~~\frac{\mbb{E}_{k}W\mathrm{e}^{-\bn{h}\cdot\bn{v}W}}{\mbb{E}_{k}\mathrm{e}^{-\bn{h}\cdot\bn{v}W}}\bigg\vert_{\bn{h}=\infty}=0,\end{equation} 
where (see \cite[Theorem 1,2]{HFM}) \begin{itemize}
	\item $\mbb{E}_kW=\infty$ if $\sup_i\mbb{E}_i|\bm{Z}_1|\log |\bm{Z}_1|=\infty$;
	\item $\mbb{E}_kW=u^{(k)}$ if $\sup_i\mbb{E}_i|\bm{Z}_1|\log |\bm{Z}_1|<\infty$.
\end{itemize}
By the definition of $b_n$, we derive the following bound for $c_{b_n}\mu^{n-b_n}/|\bm{k}_n|$: 
\begin{equation}\label{ems}
\lambda_{u}\le \frac{c_{b_n}\mu^{n-b_n}}{|\bm{k}_n|}=\frac{c_{b_{n-1}}\mu^{n-b_{n}+1}}{|\bm{k}_n|}\frac{c_{b_n}}{\mu c_{b_{n-1}}}\le \frac{\rho}{\mu},
\end{equation}
where $\lambda_{u}=(\lambda_i\min_ku^{(k)})^{-1}$. Combining this with Proposition \ref{auxlem8}, we obtain 
\begin{equation}\label{ems4}
c_{b_n}\left(\sum_{k=1}^d\hat{r}_{n-b_n,i}^{(k)}u^{(k)}\right)>\min_ku^{(k)}|\bm{\hat{r}}_{n-b_n,i}|c_{b_n}\ge |\bm{k}_n|.\end{equation}
On the other hand, by (\ref{eqlem11}), for $\bm{h}\ge C\bm{1}$ with $C$ sufficiently large, we have
\[
c_{b_n}\left(\sum_{k=1}^d\hat{r}_{n-b_n,i}^{(k)}\frac{\mbb{E}_{k}W\mathrm{e}^{-\bn{h}\cdot\bn{v}W}}{\mbb{E}_{k}\mathrm{e}^{-\bn{h}\cdot\bn{v}W}}\right)\le |\bm{k}_n|.\]
Therefore, since $\frac{\mbb{E}_{k}W\mathrm{e}^{-sW}}{\mbb{E}_{k}\mathrm{e}^{-sW}}$ is continuous with respect to $s>0$, along with the surjectivity of $v(\bm{k}):=\bm{k}\cdot\bm{v},\bm{k}\ge\bm{0},$ over $\mbb{R}_+$,  we get the existence of $\bm{h}_n$ to equation (\ref{key}). \par In addition, by Proposition \ref{auxlem8} and \eqref{ems}, we have the bound for $c_{b_n}|\bm{\hat{r}}_{n-b_n,i}|/|\bm{k}_n|$:
\begin{equation}\label{smss}
\min_ku^{(k)}\le \frac{c_{b_n}|\bm{\hat{r}}_{n-b_n,i}|}{|\bm{k}_n|}=\frac{|\bm{\hat{r}}_{n-b_n,i}|}{\mu^{n-b_n}}\frac{c_{b_n}\mu^{n-b_n}}{|\bm{k}_n|}\le\frac{\lambda'_i\rho}{\mu},~n\ge1.\end{equation} This and \eqref{key} imply that  $0<\liminf_{n\to\infty}\vt\bm{h}_n\vt \le  \limsup_{n\to\infty}\vt\bm{h}_n\vt<\infty$. The proof is complete.
	\end{proof}
	
\subsection{Proof of Theorem \ref{thm2}}
\noindent\textbf{Proof of equation \ref{ekl}:}

Suppose $\bm{k}_n\in\mbb{Z}_+^d$, $|\bm{k}_n|\ge|\hat{\bm{r}}_{n,i}|$, $|\bm{k}_n|=\mathrm{o}(c_n)$ and $\bm{k}_n\overset{\bn{v}}{\longrightarrow}\infty$ as $n\to\infty$. By Markov property, we get the decomposition: 
\begin{equation}\label{e202}
	\mbb{P}_i(\bm{Z}_{n}=\bn{k}_n)=\sum_{\bn{l}\in J_{n-b_n,i}}\mbb{P}_i(\bm{Z}_{n-b_n}=\bm{l})\mbb{P}(\bm{Z}_{b_n}=\bm{k}|\bm{Z}_0=\bm{l}).\end{equation}
Using (\ref{emm1}) and Lemma \ref{lem3}, we have the estimate for the condition probability term:  
\[
c^d_{b_n}\mbb{P}(\bm{Z}_{b_n}=\bm{k}_n|\bm{Z}_0=\bn{l})\le A(\bm{h})|\bm{l}|^{-d/2}\mathrm{e}^{\bn{h}\cdot\bn{k}_n/c_{b_n}}\left(\bm{f}_n(\mathrm{e}^{-\bn{h}/c_n})\right)^{\bn{l}}.\]
Recall that $$\hat{\bm{r}}_{n-b_n,i}=\arg \min\{\bm{s}\in J_{n-b_n,i}:|\bm{s}|\}$$ and hence $|\bm{l}|\ge|\hat{\bm{r}}_{n-b_n,i}| $ for all $\bm{l}\in J_{n-b_n,i}$. Therefore, using (\ref{smss}), we further estimate as follows:
\begin{align*}
c^d_{b_n}\mbb{P}(\bm{Z}_{b_n}=\bm{k}_n|\bm{Z}_0=\bn{l})&\le A(\bm{h})|\bm{l}|^{-d/2}\mathrm{e}^{|\bn{h}|\cdot|\bn{k}_n|/c_{b_n}}\left(\bm{f}_n(\mathrm{e}^{-\bn{h}/c_n})\right)^{\bn{l}}  \\
&\le A(\bm{h})|\bm{l}|^{-d/2}\left(\exp\left\{\frac{|\bn{h}||\bn{k}_n|}{c_{b_n}|\hat{\bn{r}}_{n-b_n,i}|}\right\}\bm{f}_n(\mathrm{e}^{-\bn{h}/c_n})\right)^{\bn{l}}\\
&\le A(\bm{h})|\bm{l}|^{-d/2}\left(\exp\left\{\frac{|\bn{h}|}{\min_{k}u^{(k)}}\right\}\bm{f}_n(\mathrm{e}^{-\bn{h}/c_n})\right)^{\bn{l}}.
\end{align*}
 By the global limit theorem, $f^{(i)}_n(\mathrm{e}^{-\bn{h}/c_n})$ converges to $\mbb{E}_iW$ as $n\to\infty$ for $1\le i \le d$. Since $\mbb{E}_iW\ge\min_{k}u^{(k)}$, there exist a vector $\bm{h}_0$ with sufficiently small $|\bm{h}_0|$ and a constant $0<\delta<1$ such that
  \begin{equation}\label{ekl3}
  	\sup_i\exp\left\{\frac{|\bn{h}_0|}{\min_{k}u^{(k)}}\right\}f^{(i)}_n(\mathrm{e}^{-\bn{h}_0/c_n})\le \mathrm{e}^{-\delta}
  \end{equation} 
holds for all $n$ sufficiently large. Consequently, we arrive at
\begin{equation}\label{ekl1}
c^d_{b_n}\mbb{P}(\bm{Z}_{b_n}=\bm{k}_n|\bm{Z}_0=\bn{l})\le A(\bm{h}_0)|\bm{l}|^{-d/2} \mathrm{e}^{-\delta|\bn{l}|}.\end{equation} 
Substituting this bound into (\ref{e202}), we derive
\[
c^d_{b_n}\mbb{P}_i(\bm{Z}_n=\bm{k}_n)\le C\cdot\mu^{-(n-b_n)d/2}f_{n-b_n}^{(i)}(\mathrm{e}^{-\delta}\bm{1}).\]
Taking the limsup of the logarithmic scaling, 
\begin{align*}\label{esm2}
	&\limsup_{n\to\infty}\mu^{b_n-n}\log[c^d_{n}\mbb{P}_i(\bm{Z}_n=\bm{k}_n)]\\
	\le& \limsup_{n\to\infty}\mu^{b_n-n}\left(\log C+\log(c^d_n/c^d_{b_n})+\log\mu^{-(n-b_n)d/2}+\log f_{n-b_n}^{(i)}(\mathrm{e}^{-\delta}\bm{1})\right) \\
	=& \limsup_{n\to\infty}\mu^{b_n-n}\log f_{n-b_n}^{(i)}(\mathrm{e}^{-\delta}\bm{1}).\end{align*}
From the logarithmic generating function limit of $\bm{f}_n$ in (\ref{e101}), we conclude
\begin{equation}\label{up1}
\limsup_{n\to\infty}\mu^{b_n-n}\log[c^d_{n}\mbb{P}_i(\bm{Z}_n=\bm{k}_n)]\le -C_{i,2},\end{equation} 
for some constant $C_{i,2}>0$.\par  Next, we establish the lower bound. Utilizing the decomposition in \eqref{e202}, we derive
 \begin{align}\label{key1}
 	\mbb{P}_i(\bm{Z}_n=\bm{k}_n)>\mbb{P}_i(\bm{Z}_{n-b_n}=\bm{\hat{r}}_{n-b_n,i})\mbb{P}(\bm{Z}_{b_n}=\bm{k}_n|\bm{Z}_0=\bm{\hat{r}}_{n-b_n,i}).
 \end{align}
 Now, we estimate the above conditional probability. From (\ref{shn1}), for any $\bm{h}\ge \bm{0}$,
 \begin{equation}\label{ems1}
	\mbb{P}(\bm{Z}_{b_n}=\bm{k}_n|\bm{Z}_0=\bm{\hat{r}}_{n-b_n,i})>\left(\bm{f}_{b_n}(\mathrm{e}^{-\bn{h}/c_{b_n}})\right)^{\hat{\bn{r}}_{n-b_n,i}}\mbb{P}(\bm{S}_{\hat{\bn{r}}_{n-b_n,i}}(\bm{h},b_n)=\bm{k}_n).\end{equation}
 By Lemma \ref{auxlem12}, there exists a sequence of $\bm{h}_n$ with $0<\inf_n\vt\bm{h}_n\vt \le  \sup_n\vt\bm{h}_n\vt<\infty$, such that $$\lim_{n\to\infty}|\mbb{E}\bm{S}_{\hat{\bn{r}}_{n-b_n,i}}(\bm{h}_n,b_n)-\bm{k}_n|=0.$$ Hence, applying Lemma \ref{auxlem11}, we obtain
 \[
 \lim_{n\to\infty}\left|\vert\hat{\bm{r}}_{n-b_n,i}\vert^{d/2}|\bm{B}^{\bn{h}_n,b_n}|^{-1}\mbb{P}(\bm{S}_{\hat{\bn{r}}_{n-b_n,i}}(\bm{h}_n,n)=\bm{k}_n)-(2\pi)^{-d/2}\right|=0.
 	\]
\par Recall that from the proof of Lemma \ref{auxlem7}, $(\bm{B}^{\bn{h},n})^2=(\bm{V}^{\bn{h},n})^{-1}$ and $c^{-2}_n\bm{V}^{\bn{h},n}$ converges to some positive definite matrix. Thus, each eigenvalue of $\bm{B}^{\bn{h}_n,n}$ is of order $c^{-1}_n$. Since the determinant equals the product of eigenvalues, we conclude $$|\bm{B}^{\bn{h}_n,b_n}|^{-1}\ge C\cdot c^d_{b_n}.$$ Combining with Proposition \ref{auxlem8}, we obtain
 	\[
 \liminf_{n \to \infty}\mu^{(n-b_n)d/2}c^d_{b_n}\mbb{P}(\bm{S}_{\hat{\bn{r}}_{n-b_n,i}}(\bm{h}_n,n)=\bm{k}_n)\ge C>0.\]
As $\inf_n\vt\bm{h}_n\vt>0$, the global limit theorem and Lemma \ref{auxlem8} imply that there exists $\theta>0$ such that 
\begin{equation}\label{ekl2}
 \left(\bm{f}_{b_n}(\mathrm{e}^{-\bn{h}_n/c_{b_n}})\right)^{\hat{\bn{r}}_{n-b_n,i}}\ge \theta^{\mu^{n-b_n}},~n\ge 1.
\end{equation}
 Substituting these bounds into (\ref{ems1}), we derive
 	\begin{equation}\label{key3}
 \liminf_{n \to \infty}\mu^{b_n-n}\log[c_n^d\mbb{P}(\bm{Z}_{b_n}=\bm{k}_n|\bm{Z}_0=\bm{s}_m)]\ge -C.\end{equation}
Next, we look at $\mbb{P}_i(\bm{Z}_{n-b_n}=\bm{\hat{r}}_{n-b_n,i})$ in (\ref{key1}). By decomposing the probability on $n-1$ generation, we have the following induction:
\begin{align*}
	\mbb{P}_i(\bm{Z}_{n}=\hat{\bm{r}}_{n,i})&=\sum_{\bn{s}\in J_{n-1,i}}\mbb{P}_i(\bm{Z}_{n-1}=\bm{s})\mbb{P}(\bm{Z}_1=\hat{\bm{r}}_{n,i}|\bm{Z}_0=\bm{s})\\
	&\ge \mbb{P}_i(\bm{Z}_{n-1}=\hat{\bm{r}}_{n-1,i})(\delta_0)^{|\hat{\bn{r}}_{n-1,i}|},
\end{align*}
where $\delta_0=\min_{1\le i\le d}\min_{\bn{s}\in\hat{J}_{1,i}}\mbb{P}_i(\bm{Z}_1=\bm{s})$. Combining this with  Proposition \ref{auxlem8} and by induction, we obtain 
	\begin{equation}\label{key4}
 \mu^{n-b_n}\log\mbb{P}(\bm{Z}_{n-b_n}=\hat{\bm{r}}_{n-b_n,i})\ge -C,~n\ge1.  	
\end{equation}
 Recalling (\ref{key3}), we get the lower bound of \eqref{key1}: 
 \[
 \liminf_{n \to \infty}\mu^{b_n-n}\log[c_n^d\mbb{P}_i(\bm{Z}_{n}=\bm{k}_n)]\ge -C_{i,1},\]
for some constant $C_{i,1}$. Together with the upper bound \eqref{up1}, this completes the proof of \eqref{ekl}. 

\noindent\textbf{Proof of equation \ref{ek2}:}

Assume $k_n\in\mbb{Z}_+$, $k_n\ge|\hat{\bm{r}}_{n,i}|$, $k_n=\mathrm{o}(c_n)$ and $k_n\to\infty$ as $n\to\infty$. Recall that $$b'_n=b(k_n\bm{1})=\min\{l\ge1 :c_l \mu^{n-l}\geq \lambda_{u}dk_n\},$$  where $d$ is the dimension parameter. Again, by Markov property, we have the decomposition: 
\begin{equation}\label{e202n}
	\mbb{P}_i(|\bm{Z}_{n}|\le k_n)=\sum_{\bn{l}\in J_{n-b'_n,i}}\mbb{P}_i(\bm{Z}_{n-b'_n}=\bm{l})\mbb{P}(|\bm{Z}_{b'_n}|\le k_n\big|\bm{Z}_0=\bm{l}).\end{equation}
Using the well-known inequality $\mbb{P}(X\le x)\le e^{hx}\mbb{E}\mbb{E}e^{-hx}$ for nonnegative random variable $X$ and $h>0$, we have
\begin{align*}
	\mathbb{P}_i(|\bm{Z}_{b'_n}| \le k_n\big|\bm{Z}_0=\bm{l})&= \mathbb{P}_i(|\bm{Z}_{b'_n}|/c_{b'_n} \le k_n/c_{b'_n}\big|\bm{Z}_0=\bm{l})
	\le \mathrm{e}^{hk_n/c_{b'_n}}(\mbb{E}\mathrm{e}^{-h|\bn{Z}_{b'_n}|/c_{b'_n}})^{|\bn{l}|}.
	\end{align*} 
Since $|\bm{l}|\ge |\hat{\bm{r}}_{n-b'_n,i}|$ for $\bm{l}\in J_{n-b'_n,i}$, applying a similar approach with the bound (\ref{smss}) yields
\begin{align*}
\mathbb{P}_i(|\bm{Z}_{b'_n}| \le k_n\big|\bm{Z}_0=\bm{l})&\le \left(\mathrm{e}^{hk_n/(|\hat{\bn{r}}_{n-b'_n,i}|c_{b'_n})}(\mbb{E}\mathrm{e}^{-h|\bn{Z}_{b'_n}|/c_{b'_n}})\right)^{|\bn{l}|}	\\
&\le\left(\mathrm{e}^{h(\min_ku^{(k)})^{-1}}(\mbb{E}\mathrm{e}^{-h|\bn{Z}_{b'_n}|/c_{b'_n}})\right)^{|\bn{l}|}.
\end{align*}
Therefore, by an argument analogous to (\ref{ekl3}), we have  
\[
\mathbb{P}_i(|\bm{Z}_{b'_n}| \le k_n\big|\bm{Z}_0=\bm{l})\le \mathrm{e}^{-\delta|\bn{l}|},~\bm{l}\in J_{n-b'_n,i}.\]
Substituting into (\ref{e202n}):
\begin{equation}
\mbb{P}_i(|\bm{Z}_{n}|\le k_n)\le f^{(i)}_{n-b'_n}(\mathrm{e}^{-\delta}\bm{1}).\end{equation}
Thus, by (\ref{e101}), we conclude 
\begin{equation}\label{up}
\mu^{b'_n-n}\log\mbb{P}_i(|\bm{Z}_{n}|\le k_n)\le -C_{i,4},\end{equation}
for some constant $C_{i,4}$.\par 
 Next, we look at the lower bound of $\mbb{P}_i(|\bm{Z}_{n}|\le k_n)$. By the decomposition in (\ref{e202n}), we get
\begin{align*}
\mbb{P}_i(|\bm{Z}_{n}|\le k_n)&\ge \mbb{P}_i(\bm{Z}_{n-b'_n}=\hat{\bm{r}}_{n-b'_n,i})\mbb{P}(|\bm{Z}_{b'_n}|\le k_n\big|\bm{Z}_0=\hat{\bm{r}}_{n-b'_n,i}).
\end{align*}
To apply Lemma \ref{auxlem12}, we construct a sequence $\bm{k}_n\in\mbb{Z}_+^d$ with components defined as: $$
k_n^{(j)}=\lfloor k_nv^{(j)}\rfloor,~1\le j\le d,$$
where $\lfloor\cdot\rfloor$ means the integer part and $v^{(j)}$ is the $j$-th component of $\bm{v}$. One can verify that  $\bm{k}_n\overset{\bn{v}}{\longrightarrow}\infty$ as $n\to\infty$. Therefore, from Lemma \ref{auxlem12}, we obtain a corresponding sequence $\bm{h}_n$ such that $\lim_{n\to\infty}|\mbb{E}\bm{S}_{\hat{\bn{r}}_{n-b'_n,i}}(\bm{h}_n,b'_n)-\bm{k}_n|=0$.  \par 
Using (\ref{shn1}) with $\bm{h}=\bm{h}_n$, we get
\[
\mbb{P}_i(|\bm{Z}_{n}|\le k_n)\ge \mbb{P}_i(\bm{Z}_{n-b'_n}=\hat{\bm{r}}_{n-b'_n,i})\left(\bm{f}_{b'_n}(\exp\{-\bm{h}_n/c_{b'_n}\})\right)^{\hat{\bm{r}}_{n-b'_n,i}}\mbb{P}(|\bm{S}_{\hat{\bn{r}}_{n-b'_n,i}}(\bm{h}_n,n)|\le k_n).\]
 By central limit theorem, we have the estimate for the last term of above equation:
\[
\lim_{n\to\infty}\mbb{P}(|\bm{S}_{\hat{\bn{r}}_{n-b'_n,i}}(\bm{h}_n,n)|\le k_n)=1/2.\]
Using (\ref{ekl2}) and (\ref{key4}) to estimate the rest terms, we derive
\[
\mu^{b'_n-n}\log\mbb{P}_i(|\bm{Z}_{n}|\le k_n)\ge -C_{i,3},\]
for some constant $C_{i,3}>0$. Combining this with the upper bound (\ref{up}), the proof is complete.

\noindent\textbf{Proof of Theorem \ref{thm2.1}:}

Based on the construction in Section \ref{sec4}, the proof of this theorem follows the same method as that of Theorem \ref{thm2}. It is worth noting that when generalizing the conclusions of Lemma \ref{auxlem7} and Lemma \ref{auxlem11}, we do not need to adopt the approach therein. Instead, we can directly apply the result from \cite[Lemmas 14 \& 15]{lower}, since in this case the random variable $S_{\bn{l},m}(h,n)$ involved is one-dimensional. We omit the detailed proof here.
\section*{Acknowledgement}
\par


\begin{thebibliography}{1}

\bibitem{Athreya1972Branching}
Athreya K B, Ney P E.
\newblock {\em Branching Processes}.
\newblock Springer Berlin Heidelberg, 1972.


\bibitem{AV99}
 Athreya K B, Vidyashankar A N. 
 \newblock Large deviation rates for branching processes. II. The multitype case.
\newblock{\em Ann. Appl. Probab.}  \textbf{5}(2):566-576, 1995.

\bibitem{lower1}
Bansaye V,  B\"{o}inghoff C. 
\newblock Lower large deviations for supercritical branching processes in random environment.
\newblock{\em  Proc. Steklov Inst. Math.} \textbf{282}:15-34, 2013.

\bibitem{gauss}
Bhattacharya, Rabi N., and R. Ranga Rao.
\newblock{\em Normal approximation and asymptotic expansions.}
\newblock Society for Industrial and Applied Mathematics, 2010.

\bibitem{2017arXiv170602919B}
Braunsteins P, Hautphenne S.
\newblock {Extinction in lower Hessenberg branching processes with countably
	many types}.
\newblock {\em Ann. Appl. Probab.} \textbf{29}(5):2782-2818, 2019.

\bibitem{lower3}
Chen X, He H.
\newblock Lower deviation and moderate deviation probabilities for maximum of a branching random walk.
\newblock{\em Ann. Inst. H. Poincaré Probab. Statist.} \textbf{56}(4):2507-2539, 2020.
\bibitem{local}
Dubuc S., Seneta.
\newblock{Local limit theorem for the Galton--Watson process.}
\newblock{\em Ann. Probab.} \textbf{4}(3): 490-496, 1976.


\bibitem{Esseen}
Esseen C G.
\newblock On the concentration function of a sum of independent random variables.
\newblock{\em Z. Wahrscheinlichkeitstheorie verw. Geb.} \textbf{9}:290-308, 1968.

\bibitem{lower}\label{lower}
Fleischmann, K., Wachtel, V.
\newblock { Lower deviation probabilities for supercritical Galton--Watson
	processes.}
\newblock{\em Ann. Inst. Henri Poincar\'e Probab. Statist.} \textbf{43}(2):233-255,2007.

\bibitem{large}\label{large}
Fleischmann, K., Wachtel, V.
\newblock {Large deviations for sums indexed by the generations of a Galton--Watson process.}
\newblock {\em Probab. Theory Related Fields.} \textbf{141}:445-470, 2008.

\bibitem{hazal}
Halász, G
\newblock Estimates for the concentration function of combinatorial number theory and probability.
\newblock{\em Period. Math. Hungar.} \textbf{8}(3):197-211, 1977.

\bibitem{large2}
He H.
\newblock On large deviation rates for sums associated with Galton--Watson processes.
\newblock {\em Adv. Appl. Probab.} \textbf{48}(3):672-690,2016.

\bibitem{HFM}\label{HFM}
Hoppe F M.
\newblock Supercritical multitype branching processes.
\newblock{\em  Ann. Probab.} \textbf{4} (3): 393-401, 1976.

\bibitem{jones}
Jones O.
\newblock{Large deviations for supercritical multitype branching processes},
\newblock{\em J. Appl. Probab.} \textbf{41}:703-720, 2004.

\bibitem{jones2}
Jones O.
\newblock{Multivariate B\"ottcher equation for polynomials with non-negative coefficients.}
\newblock{\em Aequationes Mathematicae.} \textbf{63}:251-265, 2002.

\bibitem{KA}
Kimmel M, Axelrod D E.
\newblock{\em Branching Processes in Biology}.
\newblock Springer, New York, 2002.

\bibitem{kuz}
Kuczma M. 
\newblock {\em Functional Equations in a Single Variable.}
\newblock PWN, Warszaw, 1968.

\bibitem{Kesten}
Kesten H, Stigum B P.
\newblock A limit theorem for multidimensional Galton--Watson processes.
\newblock{\em Ann. Math. Statis.} \textbf{37}:1211-1223, 1966.

\bibitem{lower2}
Sun Q, Zhang M.
\newblock Lower deviations for supercritical branching processes with immigration.
\newblock{\em Front. Math. China} \textbf{16}:567-594, 2021.

\bibitem{JR2}
Tan J, Zhang M.
\newblock Harmonic Moments for Supercritical Multi-type Galton--Watson Processes.
\newblock {\em Front. Math.}
\textbf{20}(1):215-240, 2025.



\bibliographystyle{plainnat}
\end{thebibliography}
\end{document}